\documentclass[12pt,a4paper]{article}

\usepackage[latin1]{inputenc}
\usepackage[english]{babel}
\usepackage[T1]{fontenc}
\usepackage{amsmath} 
\usepackage{amsfonts}
\usepackage{amssymb}
\usepackage{amsthm}
\usepackage{ dsfont }
\usepackage{ stmaryrd }
\usepackage{tikz}
\usepackage{todonotes}
\usetikzlibrary{matrix,arrows.meta}
\usepackage[title,titletoc,toc]{appendix}

\allowdisplaybreaks

\usepackage{hyperref}

\usepackage{xcolor} %for color
\usepackage[a4paper,top=2.1cm,bottom=2.60cm,left=2.1cm,right=2.1cm]{geometry} %toglierlo

\theoremstyle{plain}
\newtheorem{theorem}{Theorem}[section]
\newtheorem{lemma}[theorem]{Lemma}
\newtheorem{proposition}[theorem]{Proposition}

\theoremstyle{definition} 
\newtheorem{remark}[theorem]{Remark}

\begin{document}

\author{Matteo Dalla Riva\thanks{Dipartimento di Ingegneria, Universit\`a degli Studi di Palermo, Viale delle Scienze, Ed. 8, 90128 Palermo, Italy. Email: matteo.dallariva@unipa.it} ,  Paolo Luzzini\thanks{Dipartimento di Matematica `Tullio Levi-Civita', Universit\`a degli Studi di Padova, Via Trieste 63, 35121 Padova, Italy. Email: pluzzini@math.unipd.it} , Riccardo Molinarolo\thanks{Dipartimento di Scienze Molecolari e Nanosistemi, Universit\`a Ca' Foscari Venezia,  Via Torino 155, 30172 Venezia Mestre, Italy. Email: riccardo.molinarolo@unive.it} , Paolo Musolino\thanks{{Dipartimento di Matematica `Tullio Levi-Civita', Universit\`a degli Studi di Padova, Via Trieste 63, 35121 Padova, Italy. Email: paolo.musolino@unipd.it}}\ \thanks{Corresponding author}}

\title{Multi-parameter perturbations for the space-periodic heat equation}

\date{November 29, 2023}

%\date{20231129\_PerHeatPert.tex}

\maketitle

\noindent
{\bf Abstract:} This paper is divided into three parts. The first part focuses on periodic layer heat potentials, demonstrating their smooth dependence on regular perturbations of the support of integration. In the second part, we present an application of the results from the first part. Specifically, we consider a transmission problem for the heat equation in a periodic {domain} and we show that the solution depends smoothly on the shape of the transmission interface, boundary data, and {transmission} parameters. Finally, in the last part of the paper, we fix all parameters except for the {transmission parameters} and outline a strategy to deduce an explicit expansion of the solution using {Neumann}-type series.

\vspace{9pt}

\vspace{9pt}


\section{Introduction}

Understanding how the properties of {an} object depend on its shape is a crucial aspect of many real-world problems, especially when seeking to achieve the optimal configuration for maximizing some sort of efficiency.

In mathematical jargon, the quest for optimal shapes is commonly known as ``shape optimization,''  and it has garnered considerable attention in the mathematical literature.  The interested reader can find ample references and results in the monographs by  {Henrot and Pierre \cite{HePi05}  and Soko\l owski and Zol\'esio \cite{SoZo92}}.

From a mathematical standpoint, addressing such questions often involves studying how solutions to specific boundary value problems, as well as related quantities, are affected by perturbations of the domain of definition and other problem parameters. This leads us to analyze the mappings that connect a set of perturbation parameters to the solution of a boundary value problem. To undertake this {project}, having access to the toolbox of differential calculus is advantageous. Consequently, understanding the regularity properties of these maps becomes crucial. In other words, it is important to determine whether these maps are continuous, differentiable, or enjoy higher regularity properties, such as smoothness and analyticity.

These properties reveal different aspects of the perturbation and can be used in different ways: Continuity implies that small variations of the perturbation parameters correspond to small changes in the solution. Differentiability allows for characterizing the stationary points as critical points. %\todo[size=\tiny]{PM: Ho tolto `In other words, it helps identify points where the solution remains unchanged under perturbations.'} 
These critical points are important in optimization problems as they represent potential optimal configurations. Smoothness and analyticity are stronger properties. With smoothness we can approximate the solution with its Taylor expansion in the perturbation parameter with any degree of accuracy, while with  analyticity we can represent the solution as a convergent power series.

Now, a common {approach} for studying boundary value problems is {the layer potential theoretic method}, which employs integral operators to transform the original problem into a system of boundary integral equations. Eventually, this method allows us to obtain the solution as a sum of layer potentials.

As a result, an approach to understanding the perturbation sensitivity of a solution to a boundary value problem is by studying how the layer potentials and the integral operators depend upon such perturbations.

Many authors have explored this approach for elliptic equations. {For example,  Potthast \cite{Po94} proved that layer potentials for the Helmholtz equation are Fr\'echet differentiable functions of the support of integration.  Applications to scattering theory can be found, e.g., in  Haddar and Kress \cite{HaKr04} and  Kirsch \cite{Ki93}.}

However, we observe that very few results prove regularities beyond differentiability. An exception is the works of Lanza de Cristoforis and his collaborators, dedicated to proving that layer potentials and integral operators depend analytically on domain perturbations. Here we mention {Lanza de Cristoforis and Rossi \cite{LaRo04}}  for the layer potentials for the Laplace equation,  Lanza de Cristoforis and Rossi \cite{LaRo08} for the Helmholtz equation,   \cite{DaLa10} for general second order equations, and \cite{LaMu11} for the periodic case. Moreover, in \cite{DaLu23} we have obtained a smoothness result for the heat layer potentials which,  in the first part of the present paper, we will extend to the space-periodic heat layer potentials.

The method developed by Lanza de Cristoforis and collaborators was called the ``functional analytic approach'' (cf. \cite{DaLaMu21}). It was used for both regular and singular perturbations, where a perturbation is classified as regular if it does not cause any loss of regularity in the domain, and as singular if it does. 

Another approach to dealing with regular domain perturbations has recently appeared in the literature, relying on complex analysis techniques and aiming to prove the ``shape holomorphy'' of layer potential operators and integral operators. For applications of this approach, we refer the reader to {Jerez-Hanckes, Schwab, and Zech \cite{JeScZe17}, which deals with the electromagnetic wave scattering problem}.

Apart from \cite{DaLu23}, all the above cited literature concerns elliptic equations. Notably, corresponding results for parabolic problems are more scarse. To the best of our knowledge the only exceptions are { {some} works of Chapko, Kress and Yoon (see, e.g., \cite{ChKrYo98})} and Hettlich and {Rundell} \cite{HeRu01} for the Fr\'echet differentiability upon the domain of the solution of the heat equation with application to some inverse problems in heat conduction, and the already cited \cite{DaLu23} for the infinite order smoothness of the layer heat potentials upon  the support of integration.

In this paper, we adopt Lanza de Cristoforis' functional analytic approach to obtain higher order regularity results for the space-periodic version of layer heat potentials upon the support of integration. In particular, in the first part of the paper we investigate the space-periodic layer potentials for the heat equation and demonstrate that they depend smoothly on a pair $(\phi,\mu)$, where $\phi$ is a function that characterizes the shape of the domain and $\mu$ is the (pull-back of the) density function. To achieve this, we build upon similar findings for the nonperiodic heat layer potentials established in \cite{DaLu23}. To the best of our knowledge, this is the first paper to show such a result for space-periodic heat layer potentials, previous papers dealing with periodic layer potentials being dedicated to the case of {elliptic operators.}  

In the subsequent sections, we showcase how the results obtained in the first part can be utilized to examine the shape sensitivity of solutions to boundary value problems. As an illustrative application, we consider {a} transmission problem for the heat equation in a space-periodic {domain}. We show that the solution depends smoothly on the shape of the transmission interface, as well as on the boundary data and the {transmission} parameters.

Lastly, in the final part of the paper, we revisit the space-periodic transmission problem studied in the previous section. However, this time, we fix all parameters except for the { transmission parameters}. Then we outline a strategy to deduce an explicit expansion of the solution using a Neumann-type series.

The paper is organized as follows: Section \ref{s:prel} introduces some notation and preliminaries. In Section \ref{s:class}, we review certain results from \cite{DaLu23} concerning nonperiodic layer potentials. In Section \ref{s:period}, we derive analogous results for the space-periodic layer potentials. Section \ref{s:trans} investigates the perturbation sensitivity of solutions to {a} transmission problem in a space-periodic domain. Finally, in Section \ref{s:neu}, we consider the scenario where all parameters are fixed, except for the { transmission parameters}.

\section{Preliminaries}\label{s:prel}

From this point onward, we fix a value for $n$ from the set ${\mathbb{N}}\setminus\{0,1\}$, where ${\mathbb{N}}$ denotes the set of natural numbers, including zero. Additionally, we define a periodicity cell as follows:
\[
Q := \prod_{j=1}^n ]0,q_{jj}[,
\]
where $q_{jj}>0$ for all $j \in \{1,\ldots,n\}$.
We denote by $q$ the diagonal matrix 
\[q := 
 \begin{pmatrix}
  q_{11} & 0 & \cdots & 0 \\
  0 & q_{22} & \cdots & 0 \\
  \vdots  & \vdots  & \ddots & \vdots  \\
  0 & 0 & \cdots & q_{nn} 
 \end{pmatrix},
\]
and by $|Q|_n  = \prod_{i=1}^n q_{jj}$ the measure of the peridicity cell $Q$. Clearly 
\[
q\mathbb{Z}^n= \{qz:z\in \mathbb{Z}^n\}
\]
is the set of vertices of a periodic subdivision of $\mathbb{R}^n$ corresponding to the fundamental cell $Q$.  A set $A\subseteq \mathbb{R}^n$ is said to be $q$-periodic if $A +qz = A$ for all $z \in \mathbb{Z}^n$. If $A$ is a $q$-periodic set, a function $f:A \to \mathbb{R}$ is said to be $q$-periodic if $f(\cdot +qz) = f(\cdot)$ for all $z \in \mathbb{Z}^n$. 

If $\Omega$ is a subset of $\mathbb{R}^n$ then $\overline\Omega$, $\partial\Omega$, and $\nu_\Omega$ denote the closure, boundary, and, where defined, {the} outward normal to $\Omega$, respectively. If $\overline\Omega \subseteq Q$, then we set
\[
\mathbb{S}[\Omega]:=  \bigcup_{z\in \mathbb{Z}^n}(qz +\Omega)=q\mathbb{Z}^n +\Omega, \qquad
 \mathbb{S}[\Omega]^- := \mathbb{R}^n \setminus \overline{\mathbb{S}[\Omega]}.
\]
We observe that both $\mathbb{S}[\Omega]$ and $\mathbb{S}[\Omega]^-$ are $q$-periodic domains.

We will consider the heat equation
\[
\partial_t u - \Delta u = 0
\]
in domains that are space-periodic and our approach will rely on the space-periodic potential theory for the heat equation. Specifically, we will exploit space-periodic layer potentials obtained by replacing the classical fundamental solution of the heat equation with a periodic counterpart. As it is well known, a fundamental solution of the heat equation is defined as follows:
\[%\begin{equation} \label{phin}
S_{n}(t,x):=
\left\{
\begin{array}{ll}
 \frac{1}{(4\pi t)^{\frac{n}{2}} }e^{-\frac{|x|^{2}}{4t}}&{\mathrm{if}}\ (t,x)\in (0,+\infty)\times{\mathbb{R}}^{n}\,, 
 \\
 0 &{\mathrm{if}}\ (t,x)\in ((-\infty,0]\times{\mathbb{R}}^{n})\setminus\{(0,0)\}.
\end{array}
\right.
\]%\end{equation}
Then a $q$-periodic fundamental solution $S_{q,n}:(\mathbb{R}\times {\mathbb{R}}^{n})\setminus (\{0\} \times q \mathbb{Z}^n)\to\mathbb{R}$ for the heat equation is defined by taking 
\begin{equation} \label{phinq}
S_{q,n}(t,x):=
\left\{
\begin{array}{ll} 
\sum_{z\in \mathbb{Z}^n}\frac{1}{(4\pi t)^{\frac{n}{2}} }e^{-\frac{|x+qz|^{2}}{4t}}&{\mathrm{if}}\ (t,x)\in  (0,+\infty)\times{\mathbb{R}}^{n}\,, 
 \\
 0 &{\mathrm{if}}\ (t,x)\in \left((-\infty,0]\times{\mathbb{R}}^{n}\right)\setminus (\{0\} \times q \mathbb{Z}^n)
\end{array}
\right.
\end{equation}
(see  Pinsky \cite[Ch. 4.2]{Pi02} for the case $n=1$ and Bernstein, Ebert and S\"oren Krau{\ss}har \cite{BeEbSo11} for $n \geq 2$, see also \cite{Lu18}).

We will use the functional framework of Schauder classes. For the classical definitions of sets and functions belonging to class  $C^{j,\alpha}$, with $\alpha\in(0,1)$ and $j\in\{0,1\}$, we refer to Gilbarg and Trudinger \cite{GiTr83}.  For the definition of 
time-dependent functions  in the parabolic Schauder class $C^{\frac{j+\alpha}{2};j+\alpha}$ on $[0,T]\times \overline\Omega$ or $[0,T]\times \partial\Omega$  we refer to Lady\v{z}enskaja,  Solonnikov, and Ural'ceva \cite{LaSoUr68}.  In essence, a function of class $C^{\frac{j+\alpha}{2};j+\alpha}$ is 
$\left(\frac{j+\alpha}{2}\right)$-H\"older continuous in the time variable, and $(j,\alpha)$-Schauder regular in the space variable. We also denote by  $C_0^{\frac{j+\alpha}{2};j+\alpha}$ the parabolic Schauder class of functions that vanish at time $t=0$, and by $C_{0,q}^{\frac{j+\alpha}{2};j+\alpha}$ the subspace of $C_0^{\frac{j+\alpha}{2};j+\alpha}$ consisting of functions that are also $q$-periodic. The definition of parabolic Schauder classes can be extended to products of intervals and manifolds by using local charts. In the present paper {all the functional spaces we consider  { consist} of} real valued functions.

We will adopt the following notation: If $D$ is a subset of $\mathbb{R}^n$,
$T >0$ and $h$ is a map from $D$ to $\mathbb{R}^n$, we denote by $h^T$ the map from  $[0,T] \times D$
 to  $[0,T] \times \mathbb{R}^n$ defined by 
\[
h^T(t,x) := (t, h(x)) \qquad \forall (t,x) \in  [0,T] \times D.
\]
Let $\alpha\in(0,1)$ and assume that 
\begin{equation}\label{cond Omega}
\begin{split}
&\Omega\text{ is a bounded connected open subset of } \mathbb{R}^n\  \text{of class } C^{1,\alpha}\\
& \text {and has connected exterior } \Omega^-:=\mathbb{R}^n\setminus\overline{\Omega}\,. 
\end{split}   
\end{equation}
We take $\Omega$ to be the reference shape, and to formalize domain perturbations, we consider specific classes of diffeomorphisms defined on the boundary $\partial \Omega$.

Precisely, we  denote by $\mathcal{A}^{1,\alpha}_{\partial \Omega}$ the set of functions of class $C^{1,\alpha}(\partial\Omega, \mathbb{R}^{n})$ that are injective together with their differential at all points  of $\partial\Omega$.  According to Lanza de Cristoforis and Rossi \cite[Lem. 2.2, p. 197]{LaRo08} and \cite[Lem. 2.5, p. 143]{LaRo04}, $\mathcal{A}^{1,\alpha}_{\partial \Omega}$ is an open subset of $C^{1,\alpha}(\partial\Omega, \mathbb{R}^{n})$.

For $\phi \in \mathcal{A}^{1,\alpha}_{\partial \Omega}$, the Jordan-Leray separation theorem ensures that $\mathbb{R}^{n}\setminus \phi(\partial \Omega)$ has exactly two open connected components {(see, e.g.,  \cite[\S A.4]{DaLaMu21}).} We denote the bounded connected component of $\mathbb{R}^{n}\setminus \phi(\partial \Omega)$ by $\mathbb{I}[\phi]$ and the unbounded one by $\mathbb{E}[\phi]$. Moreover, we will use $\nu_\phi$ to denote the outer unit normal to $\mathbb{I}[\phi]$.
 
Then we set 
\[
{\mathcal{A}}_{\partial\Omega,Q}^{1,\alpha} := \left\{\phi \in\mathcal{A}_{\partial \Omega}^{1,\alpha} : \phi(\partial\Omega) \subseteq Q\right\}, 
\]
and for brevity, we use the notation
\[
\mathbb{S}[\phi]:= \mathbb{S}[\mathbb{I}[\phi]], \qquad \mathbb{S}[\phi]^-:= \mathbb{S}[\mathbb{I}[\phi]]^-
\]
for all $\phi \in {\mathcal{A}}_{\partial\Omega,Q}^{1,\alpha}$.
Both $\mathbb{S}[\phi]$ and $\mathbb{S}[\phi]^-$ are $q$-periodic domains depending on the diffeomorphism $\phi$ (see Figure \ref{fig:lphi}). Therefore,  we can perturb the shape of $\mathbb{S}[\phi]$ and $\mathbb{S}[\phi]^-$ by changing the function $\phi$.

\begin{figure}[!htb]
\centering
\includegraphics[width=4.2in]{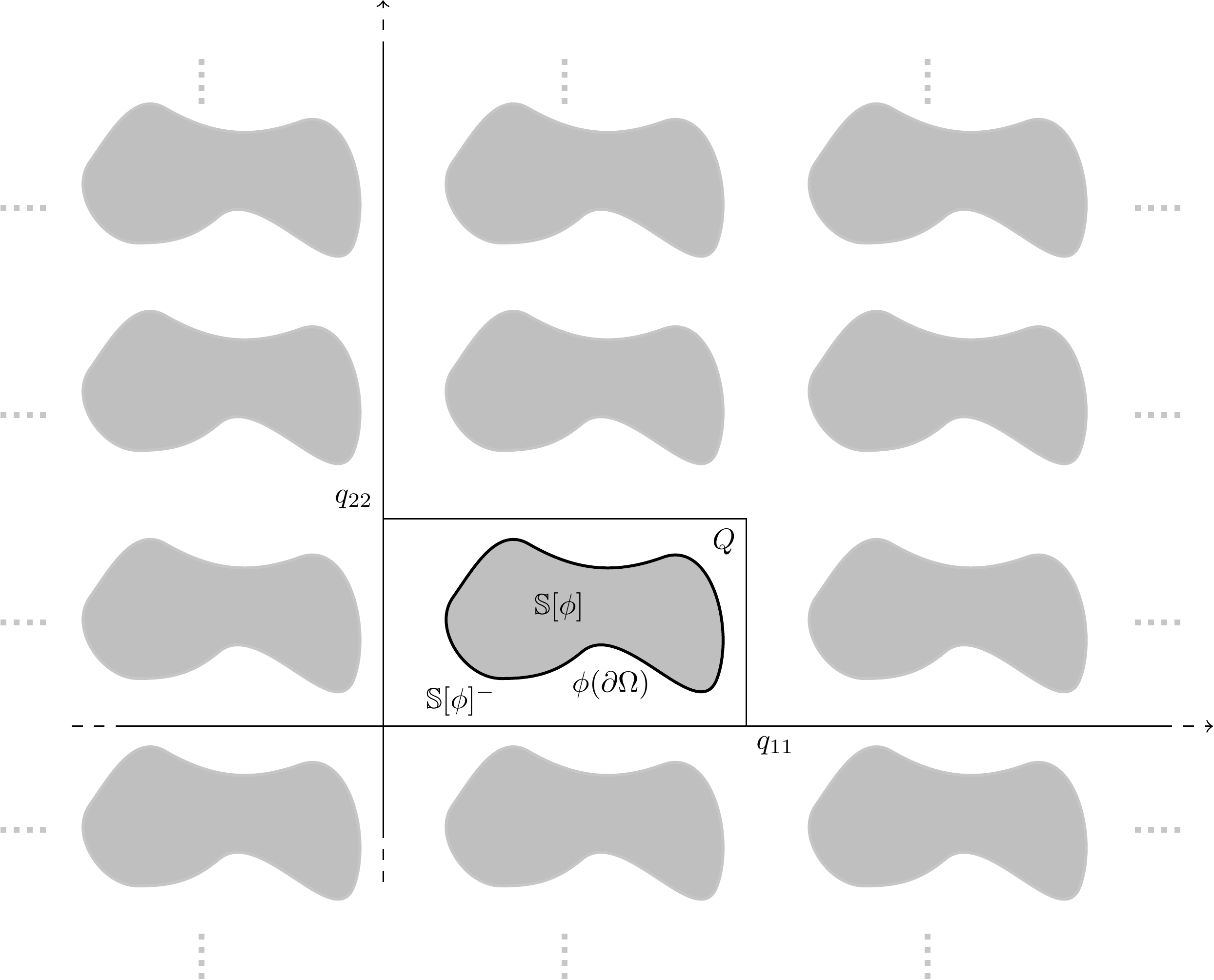}
\caption{{\it The sets $\mathbb{S}[\phi]^-$,  
$\mathbb{S}[\phi]$, and  $\phi(\partial\Omega)$ in case $n=2$.}}\label{fig:lphi} 
\end{figure}

We will consider integral operators supported on $\phi(\partial \Omega)$. To analyze their dependence on $\phi$, we will perform a change of variables. For this purpose, we rely on the following technical lemma, which shows that the map related to the change of variables in the area element and the pullback $\nu_\phi\circ\phi$ of the outer normal field depend analytically on $\phi$. A proof of this lemma can be found in Lanza de Cristoforis and Rossi \cite[p.~166]{LaRo04} and Lanza de Cristoforis \cite[Prop. 1]{La07}.
%\todo[size=\tiny]{RM: nel Lemma 2.1 mi sembra manchi l'ipotesi di regolarita' $C^{1,\alpha}$ di $\Omega$, o funziona anche senza?}
\begin{lemma}\label{rajacon}
Let $\alpha\in \mathopen (0,1)$ and  $\Omega$ be a bounded open subset of  $\mathbb{R}^n$ of class $C^{1,\alpha}$ with connected exterior.   Then the following statements hold.
\begin{itemize}
\item[(i)] For each $\phi \in \mathcal{A}^{1,\alpha}_{\partial \Omega}$, there exists a unique  
$\tilde \sigma_n[\phi] \in C^{1,\alpha}(\partial\Omega)$ such that $\tilde \sigma_n[\phi] > 0$ and 
\[ 
  \int_{\phi(\partial\Omega)}w(s)\,d\sigma_s=  \int_{\partial\Omega}w \circ \phi(y)\tilde\sigma_n[\phi](y)\,d\sigma_y, \qquad \forall w \in L^1(\phi(\partial\Omega)).
\]
Moreover, the map $\tilde \sigma_n[\cdot]$  is real analytic from $\mathcal{A}_{\partial \Omega}^{1,\alpha}  $ to $ C^{0,\alpha}(\partial\Omega)$.
\item[(ii)] The map from $\mathcal{A}_{\partial \Omega}^{1,\alpha} $ to $ C^{0,\alpha}(\partial\Omega, \mathbb{R}^{n})$ which takes $\phi$ to $\nu_{\phi} \circ \phi$ is real analytic.
\end{itemize}
\end{lemma}
%
%We turn to introduce the $\phi$-dependent periodically perforated set where our boundary value problems will be set:
%\begin{equation*}
%\mathbb{S}[\phi] := \bigcup_{z\in \mathbb{Z}^n}(qz + \mathbb{I}[\phi])=q\mathbb{Z}^n + \mathbb{I}[\phi], \qquad
% \mathbb{S}[\phi]^- := \mathbb{R}^n \setminus \overline{\mathbb{S}[\phi]}.
%\end{equation*}
%The most basic example of a boundary value problem for the periodic heat equation in 
%$[0,T] \times \mathbb{S}[\phi]^-$ is  the Dirichlet problem 
%\begin{equation}\label{prob:D}
%\begin{cases}
%\partial_t u -\Delta u = 0 \qquad &\mbox{in } ]0,T] \times \mathbb{S}[\phi]^-,\\
%u(t,x+qe_i) = u(t,x)  \qquad & \forall \, (t,x) \in [0,T] \times \overline{\mathbb{S}[\phi]^-},\, \forall\, i \in \{1,\ldots,n\}\\
%u = f \circ (\phi^T)^{(-1)} \qquad &\mbox{on } [0,T] \times \phi(\partial\Omega),\\
%u(0,\cdot) = 0  &\mbox{in }  \overline{\mathbb{S}[\phi]^-}.
%\end{cases}
%\end{equation}
%Of course it makes sense to study the dependence of the solution $u=u[\phi,f]$ of \eqref{prob:D} upon the pair $(\phi,f)$.

\section{Domain perturbations of classical layer potentials}\label{s:class}

Our first goal is to demonstrate that space-periodic layer potentials for the heat equation depend smoothly on perturbations of the support of integration. As previously mentioned in the introduction, {related} results have already been established in \cite{DaLu23} for the non-periodic layer potentials. We intend to leverage those existing results and extend them to the periodic case.

Therefore, we begin by reviewing the findings of \cite{DaLu23}, which concern layer heat potentials supported on $[0,T] \times \phi(\partial\Omega)$ for some $T>0$ and $\phi\in \mathcal{A}_{\partial \Omega}^{1,\alpha}$, as well as integral operators acting between Schauder spaces on $[0,T] \times \phi(\partial\Omega)$. However, to treat $\phi$ as a variable and state smoothness results for $\phi$-dependent functions, we need to work in a $\phi$-independent functional setting. We will then pullback the layer potentials to the fixed domain $[0,T]\times \partial \Omega$ and, simultaneously, push forward the density functions from  $[0,T]\times \partial \Omega$ to $[0,T]\times \phi(\partial\Omega)$.

To be precise, we consider the operators that take $\mu \in C_0^{\frac{\alpha}{2};\alpha}([0,T]\times \partial\Omega)$ to 
\begin{align*} 
  &V[\phi,\mu](t,\xi) := \int_0^t\int_{\phi(\partial\Omega)}  S_n(t-\tau,\phi(\xi)-y) \mu \circ (\phi^T)^{(-1)} (\tau,y) \,d\sigma_yd\tau,\\  
& V_l[\phi,\mu](t,\xi) := \int_0^t\int_{\phi(\partial\Omega)} \partial_{x_l}S_n(t-\tau,\phi(\xi)-y) \mu \circ (\phi^T)^{(-1)} (\tau,y) \,d\sigma_yd\tau \quad \forall l \in\{1,\ldots,n\},\\
&W^*[\phi,\mu](t,\xi) :=\int_0^t\int_{\phi(\partial\Omega)}D_xS_n(t-\tau,\phi(\xi)-y)\cdot \nu_{\phi}(\xi) \mu \circ (\phi^T)^{(-1)} (\tau,y) \,d\sigma_yd\tau,
\end{align*}
for all $(t,\xi) \in [0,T] \times \partial\Omega$. Additionally, for  $\psi \in C_0^{\frac{1+\alpha}{2};1+\alpha}([0,T]\times \partial\Omega)$ we define
\begin{align*} 
   &W[\phi,\psi](t,\xi) := -\int_0^t\int_{\phi(\partial\Omega)}D_xS_n(t-\tau,\phi(\xi)-y)\cdot \nu_{\phi}(y) \psi \circ (\phi^T)^{(-1)} (\tau,y) \,d\sigma_yd\tau,
\end{align*}
for all $(t,\xi) \in [0,T] \times \partial\Omega$. In the expressions above, $\partial_{x_l}S_n$ and $D_xS_n$ denote the $x_l$-derivative  and the gradient 
of $S_n$ with respect to the spatial variables, respectively. 

{The functions $V[\phi,\mu]$, $V_l[\phi,\mu]$, and $W^*[\phi,\mu]$ are the $\phi$-pullbacks of the single-layer potential and of  integral operators associated to its $x_l$-derivative and to its normal derivative. Instead $W[\phi,\psi]$  is the $\phi$-pullback of the double-layer potential.} They are defined {on $[0,T] \times \partial\Omega$} and have densities given by $\mu \circ (\phi^T)^{(-1)}$ and $\psi \circ (\phi^T)^{(-1)}$.
  
In \cite[Thm. 6.3]{DaLu23}, it has been proven that the operators $V[\phi,\cdot]$, $V_l[\phi,\cdot]$, $W^*[\phi,\cdot]$, and $W[\phi,\cdot]$ depend smoothly on the shape parameter $\phi$. Specifically, we have the following result:

\begin{theorem}\label{thm:clp}
Let $\alpha\in(0,1)$ and $T>0$. Let $\Omega$ be as in \eqref{cond Omega}.  Then, the maps that take $\phi \in \mathcal{A}^{1,\alpha}_{\partial\Omega}$ to the following operators are all of class $C^{\infty}$:
\begin{itemize}
\item[(i)]  $V[\phi,\cdot] \in \mathcal{L} \left( C^{\frac{\alpha}{2};\alpha}_0([0,T]\times\partial\Omega), C^{\frac{1+\alpha}{2};1+\alpha}_0([0,T]\times\partial\Omega) \right)$,
\item[(ii)]  $V_l[\phi,\cdot] \in \mathcal{L} \left(C^{\frac{\alpha}{2};\alpha}_0([0,T]\times\partial\Omega), C^{\frac{\alpha}{2};\alpha}_0([0,T]\times\partial\Omega) \right)$ for all $l \in \{1,\ldots,n\}$,
\item[(iii)]  $W^*[\phi,\cdot] \in \mathcal{L} \left( C^{\frac{\alpha}{2};\alpha}_0([0,T]\times\partial\Omega), C^{\frac{\alpha}{2};\alpha}_0([0,T]\times\partial\Omega)\right)$,
\item[(iv)]   $W[\phi,\cdot] \in \mathcal{L} \left( C^{\frac{1+\alpha}{2};1+\alpha}_0([0,T]\times\partial\Omega),  C^{\frac{1+\alpha}{2};1+\alpha}_0([0,T]\times\partial\Omega) \right)$\,.
\end{itemize}
\end{theorem}

Theorem \ref{thm:clp} presents an extension of similar results that were already known for layer potentials associated with elliptic equations to the parabolic setting. For example, Lanza de Cristoforis and Rossi \cite{LaRo04, LaRo08} established these results for the Laplace and Helmholtz equations, and \cite{DaLa10} for general second-order equations. However, extending these results to the parabolic setting is not a trivial task. The main difficulty lies in the interaction between the time and space variables. Applying the strategy used in \cite{LaRo04} to the parabolic case only yields a regularity result for $C^2$ perturbations of the domain, falling short of the desired $C^{1,\alpha}$ setting. 

Another difference between the elliptic and parabolic cases is that in the elliptic scenario, the layer potentials exhibit analytic dependence on the shape parameter $\phi$, while Theorem \ref{thm:clp} only guarantees that they are infinitely differentiable maps. The reason for this lack of analyticity lies in the regularity of the fundamental solution $S_n$, which is $C^\infty$ but not real analytic over the entire space $\mathbb{R}^{1+n}\setminus\{(0,0)\}$ due to its non-analytic behavior at $t=0$. In contrast, the fundamental solution of the Laplace equation, as well as other constant coefficient elliptic operators, is analytic in $\mathbb{R}^n\setminus\{0\}$.

As we shall see, such a difference implies a distinct behavior of the solutions to boundary value problems: analytic dependence on $\phi$ for the elliptic case {\it vs} {$C^\infty$-}dependence for the parabolic case.

\section{Space-periodic layer heat potentials}\label{s:period}
We now shift our focus to space-periodic layer heat potentials, where we replace the classical fundamental solution $S_n$ of the heat equation with its periodization $S_{q,n}$ (see \eqref{phinq}). We will start by introducing the definition of periodic layer potentials. Next, we will review some properties established in \cite{Lu18}. Finally, we will utilize Theorem \ref{thm:clp} to derive the corresponding regularity results for the $\phi$-pullback of periodic layer potentials.
 
Let $\alpha\in(0,1)$ and $T>0$. Let $\Omega$ be a bounded open subset of $\mathbb{R}^n$ of class $C^{1,\alpha}$ such that $\overline{\Omega} \subseteq Q$. 
For a density $\mu \in L^\infty\big([0,T] \times \partial\Omega\big)$, the $q$-periodic in space layer heat potentials are defined as
\begin{align*} 
v_q [\mu](t,x) 
 := \int_{0}^{t} \int_{\partial \Omega} S_{q,n}(t-\tau,x-y) \mu(\tau, y)\,d\sigma_y d\tau \qquad \forall\,(t,x) \in [0, T] \times \mathbb{R}^n,
\end{align*}
and 
\begin{equation*} 
w_q[\mu](t,x) := -\int_{0}^t\int_{\partial\Omega}  D_x
S_{q,n}(t-\tau,x-y)\cdot \nu_\Omega(y) \mu(\tau,y)\,d\sigma_yd\tau \qquad \forall\,(t,x) \in [0,T]\times \mathbb{R}^n.
\end{equation*}
The functions $v_q[\mu]$ and $w_q[\mu]$ are called respectively the $q$-periodic single- and double-layer heat potential.
 Moreover, we set
\begin{align*}
w^*_{q}[\mu](t,x) := \int_{0}^t\int_{\partial\Omega} D_x
S_{q,n}(t-\tau,x-y)\cdot  \nu_\Omega(x)\mu(\tau,y)\,d\sigma_yd\tau \qquad \forall\,(t,x) \in [0,T] \times \partial\Omega.
\end{align*}
The map $w^*_{q}[\mu]$ is related to the normal
derivative of the $q$-periodic in space single-layer potential (see Theorem \ref{thmsl}).

Periodic layer heat potentials enjoy properties similar to that of their standard counterpart. We collect them in the following two theorems. The proofs  can be  found in   \cite[Thms. 2, 3]{Lu18}.
\begin{theorem}\label{thmsl}
Let $\alpha\in(0,1)$ and $T>0$. Let $\Omega$ be a bounded open subset of $\mathbb{R}^n$ of class $C^{1,\alpha}$ such that $\overline{\Omega} \subseteq Q$. Then the following statements hold.
\begin{itemize}

\item[(i)] Let $\mu \in L^\infty([0,T] \times \partial\Omega)$. Then $v_q[\mu]$  is continuous, $q$-periodic in space and 
$v_q[\mu] \in C^\infty\big({(0,T]} \times (\mathbb{R}^n \setminus \partial\mathbb{S}[\Omega])\big)$. 
Moreover $v_q[\mu]$ solves the heat equation 
in $(0,T]\times (\mathbb{R}^n \setminus \partial\mathbb{S}[\Omega])$.

\item[(ii)] Let $v_q^+[\mu]$ and $v_q^-[\mu]$ denote 
the restrictions of $v_q[\mu]$ to $[0,T] \times \overline{\mathbb{S}[\Omega]}$ and to $[0,T]\times \overline{\mathbb{S}[\Omega]^-}$, respectively. The map from  $C_0^{\frac{\alpha}{2};  \alpha}([0,T] \times \partial\Omega)$ to  $C_{0,q}^{\frac{1+\alpha}{2}; 1+\alpha}\big([0,T] \times \overline{\mathbb{S}[\Omega]}\big)$ that takes $\mu$ to $v_q^+[\mu]$ is linear and continuous. Likewise, the map from  $C_0^{\frac{\alpha}{2};  \alpha}([0,T] \times \partial\Omega)$ to  $C_{0,q}^{\frac{1+\alpha}{2}; 1+\alpha}\big([0,T] \times \overline{\mathbb{S}[\Omega]^-}\big)$ that {takes} $\mu$ with $v_q^-[\mu]$ is also linear and continuous.

\item[(iii)] Let $\mu \in C_0^{\frac{\alpha}{2};  \alpha}([0,T] \times \partial\Omega)$ and $l \in \{1,\ldots,n\}$. Then the following jump relations hold:
\begin{align*} %\label{jumpsl}
\frac{\partial}{\partial \nu_\Omega}v_q^\pm[\mu](t,x)  = &\pm \frac{1}{2}\mu(t,x)
 +w_{q}^*[\mu](t,x),\\ \nonumber
\partial_{x_l}v_q^\pm[\mu](t,x) = &\pm \frac{1}{2}\mu(t,x)\left(\nu_{\Omega}(x)\right)_l + \int_{0}^t \int_{\partial\Omega} 
\partial_{x_l} S_{q,n}(t-\tau,x-y)\mu(\tau,y)\, d\sigma_yd\tau, \nonumber
\end{align*}
for all $(t,x) \in [0,T] \times \partial\Omega$.

\end{itemize}
\end{theorem}
\begin{theorem} \label{thmdl}
Let $\alpha\in(0,1)$ and $T>0$. Let $\Omega$ be a bounded open subset of $\mathbb{R}^n$ of class $C^{1,\alpha}$ such that $\overline{\Omega} \subseteq Q$. Then the following statements hold.
\begin{itemize}
\item[(i)] Let $\mu \in L^\infty([0,T] \times \Omega)$. Then $w_q[\mu]$  is $q$-periodic in space, 
$w_q[\mu] \in C^\infty\big({(0,T]} \times (\mathbb{R}^n \setminus \partial\mathbb{S}[\Omega])\big)$, and $w_q[\mu]$ solves the heat equation in 
$(0,T] \times(\mathbb{R}^n \setminus \partial\mathbb{S}[\Omega])$.

\item[(ii)] Let $\mu \in C_0^{\frac{1+\alpha}{2};  1+\alpha}([0,T] \times \partial\Omega)$.  Then the restriction $w_q[\mu]_{|[0,T] \times\mathbb{S}[\Omega]}$ can be 
extended uniquely to an element $w_q^+[\mu] \in C_{0,q}^{\frac{1+\alpha}{2};1+\alpha}\big([0,T]\times \overline{\mathbb{S}[\Omega]}\big)$ and 
the restriction $w_q[\mu]_{|[0,T] \times \mathbb{S}[\Omega]^-}$ can be 
extended uniquely to an element $w_q^-[\mu] \in C_{0,q}^{\frac{1+\alpha}{2};1+\alpha}\big([0,T] \times\overline{\mathbb{S}[\Omega]^-}\big)$.  
Moreover the following jump formulas hold:
\begin{align*} 
&w_q^\pm[\mu](t,x) = \mp \frac{1}{2} \mu(t,x) + w_q[\mu](t,x)\,,
\\  
&\frac{\partial}{\partial \nu_\Omega}w_q^+[\mu](t,x) - \frac{\partial}{\partial \nu_\Omega}w_q^-[\mu](t,x) = 0,
\end{align*}
for all $(t,x) \in [0,T] \times \partial\Omega$.

\item[(iii)] The map from $C^{\frac{1+\alpha}{2};1+\alpha}_0([0,T] \times \partial\Omega)$ to $C_{0,q}^{\frac{1+\alpha}{2};1+\alpha}\big([0,T]\times \overline{\mathbb{S}[\Omega]}\big)$ 
that takes $\mu$ to the function $w_q^+[\mu]$ is linear and continuous. Likewise, the map from $C^{\frac{1+\alpha}{2};1+\alpha}_0([0,T] \times \partial\Omega)$ to $C_{0,q}^{\frac{1+\alpha}{2};1+\alpha}\big([0,T]\times \overline{\mathbb{S}[\Omega]^-}\big)$ 
{that} takes $\mu$ to the function $w_q^-[\mu]$ is also linear and continuous. 
\end{itemize}
\end{theorem}
% Or first result on the dependence of periodic layer heat potentials shows that the pulled back layer potentials on $\Omega_{\omega,\delta}$ depends  smoothly on the domain parametrization.
% \begin{theorem}\label{sslhp}
% Let $\alpha$, $\Omega$, $\omega$, $\delta_\Omega$ as in Lemma \ref{ext1}, 
% $\delta \in \mathopen ]0,\delta_\Omega[$,   $T  > 0$.  Then the following 
% statements hold.
% \begin{itemize} 
% \item[(i)] The maps $V_q^\pm$ defined by
% \[
% V_q^\pm[\Phi,\mu] := v_q^\pm[\mu\circ(\Phi^T)^{(-1)}]\circ\Phi^T \qquad \mbox{ in } [0,T]\times\overline{\Omega_{\omega,\delta}^\pm}
% \]
% are of class $C^\infty$ from $(C^{1,\alpha}(\overline{ \Omega_{\omega,\delta}}) \cap \mathcal{A}'_{\overline{\Omega_{\omega,\delta}}}) \times C^{\frac{\alpha}{2};\alpha}_0([0,T]\times\partial\Omega)$ 
% to $C^{\frac{1+\alpha}{2};1+\alpha}_0\Big([0,T]\times\overline{\Omega^\pm_{\omega,\delta}}\Big)$.
% \item[(ii)] The maps $W_q^\pm$ defined by
% \[
% W_q^\pm[\Phi,\mu] := w_q^\pm[\mu\circ(\Phi^T)^{(-1)}]\circ\Phi^T  \qquad \mbox{ in } [0,T]\times\overline{\Omega_{\omega,\delta}^\pm}
% \]
% are of class $C^\infty$ from $(C^{1,\alpha}(\overline{ \Omega_{\omega,\delta}}) \cap \mathcal{A}'_{\overline{\Omega_{\omega,\delta}}}) \times C^{\frac{1+\alpha}{2};1+\alpha}_0([0,T]\times\partial\Omega)$ 
% to $C^{\frac{1+\alpha}{2};1+\alpha}_0\Big([0,T]\times\overline{\Omega^\pm_{\omega,\delta}}\Big)$.
% \end{itemize}
%\end{theorem}

The main idea in the proof of Theorems \ref{thmsl} and \ref{thmdl} revolves around representing periodic layer potentials as the sum of their non-periodic counterparts and a remainder, which is an integral operator with a nonsingular kernel. This is feasible because the map
\begin{equation}\label{def:R}
R_{q,n}(t,x) := S_{q,n}(t,x) - S_n(t,x), \qquad\forall \,(t,x) \in ({\mathbb{R}}\times {\mathbb{R}}^{n})\setminus (\{0\} \times q \mathbb{Z}^n)
\end{equation}
can be extended by continuity to
$({\mathbb{R}}\times {\mathbb{R}}^{n})\setminus (\{0\} \times q (\mathbb{Z}^n \setminus\{0\}))$. Keeping the notation  $R_{q,n}$  for this extension, we have that
\[
R_{q,n} \in C^\infty (({\mathbb{R}}\times {\mathbb{R}}^{n})\setminus  (\{0\} \times q (\mathbb{Z}^n \setminus\{0\}))).
\]
In other words, $R_{q,n}$  is smooth in a neighborhood of the origin $(0,0)$. A proof of this assertion can be found in \cite[Thm. 1]{Lu18}.

The same idea can be used to 
recover the periodic counterpart of Theorem \ref{thm:clp}. We first need to introduce the pull-back of the boundary integral operators associated with $q$-periodic layer heat potentials. Let $\Omega$ be a bounded open subset of $\mathbb{R}^n$ of class $C^{1,\alpha}$ such that both 
$\Omega$ and $\Omega^-$ are connected. Let $\phi \in {\mathcal{A}}_{\partial\Omega,Q}^{1,\alpha}$. For
$\mu \in C_0^{\frac{\alpha}{2};\alpha}([0,T]\times \partial\Omega)$, we consider the operators
\begin{align*}
&V_q[\phi,\mu](t,\xi) :=  \int_0^t\int_{\phi(\partial\Omega)}  S_{q,n}(t-\tau,\phi(\xi)-y) \mu \circ (\phi^T)^{(-1)} (\tau,y) \,d\sigma_yd\tau  \\
& V_{q,l}[\phi,\mu](t,\xi) := \int_0^t\int_{\phi(\partial\Omega)} \partial_{x_l}S_{q,n}(t-\tau,\phi(\xi)-y) \mu \circ (\phi^T)^{(-1)} (\tau,y) \,d\sigma_yd\tau \quad \forall l \in\{1,\ldots,n\}\\
&W^*_{q}[\phi,\mu](t,\xi) :=\int_0^t\int_{\phi(\partial\Omega)}D_xS_{q,n}(t-\tau,\phi(\xi)-y)\cdot \nu_{\phi}(\xi) \mu \circ (\phi^T)^{(-1)} (\tau,y) \,d\sigma_yd\tau,
\end{align*}
for all $(t,\xi) \in [0,T] \times \partial\Omega$. Also, for $\psi \in C_0^{\frac{1+\alpha}{2};1+\alpha}([0,T]\times \partial\Omega)$ we set
\begin{align*}
     &W_q[\phi,\psi](t,\xi) := -\int_0^t\int_{\phi(\partial\Omega)}D_xS_{q,n}(t-\tau,\phi(\xi)-y)\cdot \nu_{\phi}(y) \psi \circ (\phi^T)^{(-1)} (\tau,y) \,d\sigma_yd\tau,
\end{align*}
for all $(t,\xi) \in  [0,T] \times \partial\Omega$. 
Similarly to the non-periodic scenario,  
the  function $V_q[\phi,\mu]$ is the $\phi$-pullback of the $q$-periodic single-layer potential restricted on the boundary  $[0,T] \times \phi(\partial\Omega)$, while $V_{q,l}[\phi,\mu]$ and $W^*_{q}[\phi,\mu]$ are respectively related to its $x_l$ and normal derivatives. The function $W_q[\phi,\psi]$ is instead related to the boundary behavior of the $q$-periodic double-layer potential.

We are now ready to present the main result of this section, concerning the smoothness of the mappings that associate $\phi$ with $V_q[\phi,\cdot]$, $V_{q,l}[\phi,\cdot]$, $W^*_{q}[\phi,\cdot]$, and $W_{q}[\phi,\cdot]$.

\begin{theorem}\label{thm:main}
Let $\alpha\in(0,1)$ and $T>0$. Let $\Omega$ be as in \eqref{cond Omega}.  
Then the maps that take $\phi \in \mathcal{A}^{1,\alpha}_{\partial\Omega,Q}$ to the following operators are all of class $C^\infty$:
 \begin{itemize}
  \item[(i)]  $V_q[\phi,\cdot] \in \mathcal{L} \left(C^{\frac{\alpha}{2};\alpha}_0([0,T]\times\partial\Omega), C^{\frac{1+\alpha}{2};1+\alpha}_0([0,T]\times\partial\Omega)\right)$,
  \item[(ii)]  $V_{q,l}[\phi,\cdot] \in \mathcal{L} \left( C^{\frac{\alpha}{2};\alpha}_0([0,T]\times\partial\Omega), C^{\frac{\alpha}{2};\alpha}_0([0,T]\times\partial\Omega)\right)$ 
  for all $l \in \{1,\ldots,n\}$,
  \item[(iii)]  $W^*_{q}[\phi,\cdot] \in \mathcal{L}\left( C^{\frac{\alpha}{2};\alpha}_0([0,T]\times\partial\Omega), C^{\frac{\alpha}{2};\alpha}_0([0,T]\times\partial\Omega)\right)$,
  \item[(iv)]   $W_q[\phi,\cdot] \in \mathcal{L} \left(C^{\frac{1+\alpha}{2};1+\alpha}_0([0,T]\times\partial\Omega), C^{\frac{1+\alpha}{2};1+\alpha}_0([0,T]\times\partial\Omega)\right)$\,.
 \end{itemize}
 \end{theorem}
\begin{proof}
We confine ourselves to demonstrate the theorem for the map $\phi\mapsto V_q[\phi,\cdot]$ in point (i).  The proof for the operators in (ii), (iii), and (iv) can be carried out by a straightforward adaptation of the argument presented below.   In these cases, we will use statements (ii), (iii), and (iv) of Theorem \ref{thm:clp}, analogously to how we will use statement (i) of the same Theorem \ref{thm:clp} in the forthcoming argument.

As shown in \cite[Thm. 1]{Lu18}, the map $R_{q,n}$ defined in \eqref{def:R} is of class $C^\infty$ in the set
$({\mathbb{R}}\times {\mathbb{R}}^{n})\setminus  (\{0\} \times q (\mathbb{Z}^n \setminus\{0\}))$.
In particular, $R_{q,n}$  is smooth in a neighborhood of  $(0,0)\in \mathbb{R}\times \mathbb{R}^n$.

Let $(\phi, \mu) \in  {\mathcal{A}}_{\partial\Omega,Q}^{1,\alpha} \times C^{\frac{\alpha}{2};\alpha}_0([0,T]\times\partial\Omega)$. Clearly, definition \eqref{def:R} implies that
\begin{align}\label{thm:main1}
V_q[\phi,\mu](t,\xi) = V[\phi,\mu](t,\xi) + \int_0^t\int_{\phi(\partial\Omega)}R_{q,n}(t-\tau,\phi(\xi)-y) \mu \circ (\phi^T)^{(-1)} (\tau,y) \,d\sigma_yd\tau
\end{align}
for all $(t,\xi) \in [0,T] \times \partial \Omega$. By Theorem \ref{thm:clp} (i), the map that takes $\phi \in {\mathcal{A}}_{\partial\Omega,Q}^{1,\alpha}$ to  
\[
V[\phi,\cdot] \in \mathcal{L} \left( C^{\frac{\alpha}{2};\alpha}_0([0,T]\times\partial\Omega), C^{\frac{1+\alpha}{2};1+\alpha}_0([0,T]\times\partial\Omega) \right)
\]
is of class $C^\infty$.
We now consider the second term on the right-hand side of \eqref{thm:main1}. 
By Lemma \ref{rajacon} we have
\begin{equation*}
\begin{split}
\int_0^t\int_{\phi(\partial\Omega)}&R_{q,n}(t-\tau,\phi(\xi)-y) \mu \circ (\phi^T)^{(-1)} (\tau,y) \,d\sigma_y d\tau
\\
&=  \int_0^t\int_{\partial\Omega}R_{q,n}(t-\tau,\phi(\xi)-\phi(\eta)) \mu   (\tau,y) \tilde \sigma_n[\phi](\eta)\,d\sigma_\eta d\tau.
\end{split}
\end{equation*}
We note that 
 \[
\phi(\xi)-\phi(\eta) \notin q\mathbb{Z}^n \setminus \{0\}\qquad \forall\,(\xi,\eta) \in \partial  \Omega \times \partial \Omega.
 \]
Indeed, if it was that 
$(\xi,\eta) \in   \partial  \Omega \times \partial \Omega$ and $\phi(\xi)-\phi(\eta) \in q\mathbb{Z}^n\setminus\{0\}$, then we would have that
$\phi(\xi) \in \phi(\partial \Omega) + q\mathbb{Z}^n\setminus\{0\}$, which clearly cannot be. Then, by Lemma \ref{rajacon} and by  the results of 
\cite[Lemma A.2, Lemma A.3]{DaLu23} on non-autonomous composition operators and on time-dependent integral operators with non-singular kernels, we deduce that 
the map from ${\mathcal{A}}_{\partial\Omega,Q}^{1,\alpha} \times C^{\frac{\alpha}{2};\alpha}_0([0,T]\times\partial\Omega)$ to $C^{\frac{1+\alpha}{2};1+\alpha}_0([0,T]\times\partial\Omega)$ that takes $(\phi,\mu)$  to the function
\begin{equation*}
    K[\phi,\mu](t,\xi) := \int_0^t\int_{\partial\Omega}R_{q,n}(t-\tau,\phi(\xi)-\phi(\eta)) \mu   (\tau,y) \tilde \sigma_n[\phi](\eta)\,d\sigma_\eta d\tau \qquad \forall (t,\xi) \in [0,T] \times \partial\Omega,
\end{equation*}
is of class $C^\infty$. 

It remains to show that $\phi\mapsto K[\phi,\cdot]$ is $C^\infty$ from ${\mathcal{A}}_{\partial\Omega,Q}^{1,\alpha}$ to the operator space
\[
\mathcal{L} \left( C^{\frac{\alpha}{2};\alpha}_0([0,T]\times\partial\Omega), C^{\frac{1+\alpha}{2};1+\alpha}_0([0,T]\times\partial\Omega) \right)\,.
\]
Given that $K[\phi,\mu]$
is linear and continuous with respect to the variable $\mu$, we have
\begin{equation}\label{thm:main2}
    K[\phi,\cdot]= d_{\mu} K[\phi,\mu] \qquad \forall (\phi,\mu) \in {\mathcal{A}}_{\partial\Omega,Q}^{1,\alpha} \times C^{\frac{\alpha}{2};\alpha}_0([0,T]\times\partial\Omega),
\end{equation}
where the term on the right-hand side is the partial Frechet differential of $(\phi,\mu)\mapsto K[\phi,\mu]$ with respect to $\mu$, evaluated at the point $(\phi,\mu)$. Because $(\phi,\mu)\mapsto K[\phi,\mu]$ is a map of class $C^\infty$, the map that takes $(\phi,\mu)$ to $d_{\mu} K[\phi,\mu]$ is also of class $C^\infty$ from ${\mathcal{A}}_{\partial\Omega,Q}^{1,\alpha} \times C^{\frac{\alpha}{2};\alpha}_0([0,T]\times\partial\Omega)$ to the operator space $\mathcal{L} \left( C^{\frac{\alpha}{2};\alpha}_0([0,T]\times\partial\Omega), C^{\frac{1+\alpha}{2};1+\alpha}_0([0,T]\times\partial\Omega) \right)$. Hence, the map $(\phi,\mu) \mapsto K[\phi,\cdot]$ is  of class $C^\infty$ by \eqref{thm:main2}, and, since it does not depend on $\mu$, we conclude that  $\phi\mapsto K[\phi,\cdot]$ is $C^\infty$ from ${\mathcal{A}}_{\partial\Omega,Q}^{1,\alpha}$ to the operators space $\mathcal{L} \left( C^{\frac{\alpha}{2};\alpha}_0([0,T]\times\partial\Omega), C^{\frac{1+\alpha}{2};1+\alpha}_0([0,T]\times\partial\Omega) \right)$. 

Hence, the validity of the theorem for the map $\phi\mapsto V[\phi,\cdot]$ in point (i) has  now been proven.
\end{proof}

It is worth recalling that a result similar to Theorem \ref{thm:main} {has been previously  proven} in \cite{LaMu11} for periodic layer potentials corresponding to a general class of second-order elliptic equations. Later, these findings were used to study the effect of perturbations on physical quantities relevant to materials science and fluid mechanics. For instance, {we refer to \cite{DaLuMuPu22} which deals with the effective properties of periodic structures.}

 %%%%%%%%%%%%%%%%%%%%%%%%%%%%%%%%%%%%%%%%%%%%%%%
\section{{A  transmission} problem} \label{s:trans}

The theorem presented in the preceding section, Theorem \ref{thm:main}, serves as a toolkit to analyze the solution to boundary value problems for the heat equation in spatially periodic domains. The primary goal of using this theorem is to demonstrate the smooth dependence of such solutions on shape perturbations. As emphasized in the introduction, the feasibility of employing Theorem \ref{thm:main} for this purpose relies on the applicability of boundary integral operators and layer potentials to derive solutions for boundary value problems. 

As an illustrative application, we consider a {periodic  transmission} problem. We will demonstrate that its solution depends smoothly on the shape of the transmission interface, the boundary data, and the {transmission} parameters. 

Now, let's introduce this specific problem. Consider $\alpha \in (0,1)$, $T>0$, and a bounded open subset $\Omega$ of $\mathbb{R}^n$ of class $C^{1,\alpha}$ such that both $\Omega$ and its exterior $\Omega^-$ are connected. Let  $\phi \in {\mathcal{A}}_{\partial\Omega,Q}^{1,\alpha}$.
We fix   {the transmission parameters} $\lambda^+,\lambda^- > 0$ and choose  $f \in C^{\frac{1+\alpha}{2};1+\alpha}_0([0,T]\times\partial \Omega)$ and 
$g \in C^{\frac{\alpha}{2};\alpha}_0([0,T]\times \partial\Omega)$. With this setup, we proceed to consider the following {transmission} problem:
\begin{align}\label{periodicidtran}
\left\{
\begin{array}{ll}
\partial_t u^+ - \Delta u^+ = 0 &\mbox{in }  (0, T]\times \mathbb{S}[\phi], \\
\partial_t u^- - \Delta u^- = 0 &\mbox{in }  (0, T]\times \mathbb{S}[\phi]^-, \\
u^+(t,x+qz) = u^+(t,x) &\forall\,(t,x) \in  [0,T] \times \overline{  \mathbb{S}[\phi]},\, \forall\, z\in \mathbb{Z}^n,  \\
u^-(t,x+qz) = u^-(t,x) &\forall\,(t,x) \in  [0,T] \times   \overline{\mathbb{S}[\phi]^-}, \, \forall\,z\in \mathbb{Z}^n, \\
u^+-u^-= f \circ (\phi^T)^{(-1)}&\mbox{on } [0, T]\times \partial \Omega,\\
\lambda^-\frac{\partial} {\partial \nu_\Omega} u^- - \lambda^+\frac{\partial} {\partial \nu_\Omega} u^+ = g\circ (\phi^T)^{(-1)} &\mbox{on } [0, T]\times \partial \Omega,\\
u^+(0,\cdot) = 0 &\mbox{in } \overline{ \mathbb{S}[\phi]},\\
u^-(0,\cdot) = 0 &\mbox{in }  \overline{ \mathbb{S}[\phi]^-}.
\end{array}
\right.
\end{align}
{Problem \eqref{periodicidtran} can be seen as the periodic version in $(0, T]\times \mathbb{S}[\phi]$ and $(0, T]\times \mathbb{S}[\phi]^-$ of the transmission problem for the heat equation considered  in {Hofmann}, Lewis, and Mitrea \cite{HoLeMi03}. We emphasize that there are other transmission problems for the heat equation that are relevant in applications, and in particular we refer to the one considered in Qiu, Rieder, Sayas, and Zhang \cite{QiRiSaZh19}.}

In \cite[Thm. 4]{LuMu18} it { has been proved} that the solution $(u^+,u^-)$ of \eqref{periodicidtran} exists, is unique, and belongs to a suitable product of  
Schauder spaces. Moreover, this solution can be expressed as a {pair} of periodic single-layer heat potentials, and the densities of these potentials are solutions to a particular system of boundary integral equations. To be precise, the following result holds:
\begin{theorem}\label{thm:uniqsol}
Let $\alpha\in(0,1)$ and $T>0$. Let $\Omega$ be as in \eqref{cond Omega}.
Let  $\phi \in {\mathcal{A}}_{\partial\Omega,Q}^{1,\alpha}$.
Let  $\lambda^+,\lambda^- > 0$ and  $f \in C^{\frac{1+\alpha}{2};1+\alpha}_0([0,T]\times\partial \Omega)$, 
$g \in C^{\frac{\alpha}{2};\alpha}_0([0,T]\times \partial\Omega)$. Then problem \eqref{periodicidtran} has
a unique solution
\[(u^+,u^-) \in C^{\frac{1+\alpha}{2};1+\alpha}_{0,q}([0,T]\times\overline{\mathbb{S}[\phi]}) \times 
C^{\frac{1+\alpha}{2};1+\alpha}_{0,q}([0,T]\times\overline{\mathbb{S}[\phi]^-}).\]
Moreover,
\begin{align*}
u^+ = v^+_q[\mu^+], \qquad u^-=v^-_q[\mu^-],
\end{align*}
where {$(\mu^+,\mu^-)$  is the unique solution in $C^{\frac{\alpha}{2};\alpha}_{0}([0,T] \times \phi(\partial\Omega)) \times C^{\frac{\alpha}{2};\alpha}_{0}([0,T] \times \phi(\partial\Omega))$} of the system of integral equations
\begin{equation}\label{sys1}
\begin{cases}
  v_q^+[\mu^+]_{|[0,T]\times\phi(\partial\Omega)} - v_q^-[\mu^-]_{|[0,T]\times\phi(\partial\Omega)} =f \circ (\phi^T)^{(-1)},\\
  \lambda^-\left(-\frac{1}{2}\mu^- + w^*_{q}[\mu^-]\right)  - 
 \lambda^+\left(\frac{1}{2}\mu^+ + w^*_{q}[\mu^+]\right)=g \circ (\phi^T)^{(-1)}.
\end{cases}
\end{equation}

  \end{theorem}
Keeping in mind Theorem \ref{thm:uniqsol}, we will use the notation
 \[
 (u^+[\phi,\lambda^+,\lambda^-,f,g], u^-[\phi,\lambda^+,\lambda^-,f,g])
\]
to denote the unique solution of problem \eqref{periodicidtran}.  
  
Moreover, thanks to Theorem \ref{thm:uniqsol}, we have a representation of the unique solution of the transmission problem as a pair of single-layer potentials with densities that solve the system of boundary integral equations in \eqref{sys1}. Then, to understand how the solution 
depends upon variations of $\phi$, $\lambda^+$, $\lambda^-$, $f$, and $g$, we plan to first understand how the densities depend on such parameters. To maintain consistency within the functional spaces, we have to perform a $\phi$-pullback of the integral equations in \eqref{sys1}. This transformation results in a system of $\phi$-dependent integral equations defined on the fixed domain $[0,T] \times \partial \Omega$. This is achieved through a change of variables applied to \eqref{sys1}, leading to the following proposition:
\begin{proposition}\label{periodiccp}
Let $\alpha\in(0,1)$ and $T>0$. Let $\Omega$ be as in \eqref{cond Omega}. Let  $\phi \in {\mathcal{A}}_{\partial\Omega,Q}^{1,\alpha}$.
Let  $\lambda^+,\lambda^- > 0$ and  $f \in C^{\frac{1+\alpha}{2};1+\alpha}_0([0,T]\times\partial \Omega)$, 
$g \in C^{\frac{\alpha}{2};\alpha}_0([0,T]\times \partial\Omega)$.  
Then the unique  solution 
\begin{equation*}
(u^+[\phi,\lambda^+,\lambda^-,f,g], u^-[\phi,\lambda^+,\lambda^-,f,g]) \in C^{\frac{1+\alpha}{2};1+\alpha}_{0,q}([0,T]\times\overline{\mathbb{S}[\phi]}) \times 
C^{\frac{1+\alpha}{2};1+\alpha}_{0,q}([0,T]\times\overline{\mathbb{S}[\phi]^-})    
\end{equation*}
of problem \eqref{periodicidtran} can
be written as 
\[
u^+[\phi,\lambda^+,\lambda^-,f,g] = v_q^+[\rho^+ \circ (\phi^T)^{(-1)}] \qquad u^-[\phi,\lambda^+,\lambda^-,f,g]=v^-_q[\rho^- \circ (\phi^T)^{(-1)}],
\]
where { $(\rho^+,\rho^-)$ is the unique solution in  $C_0^{\frac{\alpha}{2};\alpha}([0,T] \times \partial\Omega) \times C_0^{\frac{\alpha}{2};\alpha}([0,T] \times \partial\Omega)$}  of the system of integral equations
\begin{equation}\label{sys2}
\begin{cases}
  V_q[\phi,\rho^+]  - V_q[\phi,\rho^-]  =f  ,
  \\
  \lambda^-\left(-\frac{1}{2}\rho^- + W_{q}^*[\phi,\rho^-]\right)  - 
  \lambda^+\left(\frac{1}{2}\rho^+ + W_{q}^*[\phi,\rho^+]\right)=g .
\end{cases}
\end{equation}
\end{proposition}

Our next step is to understand the dependence of the solution $(\rho^+,\rho^-) $ of \eqref{sys2} upon $(\phi,\lambda^+,\lambda^-,f,g)$. To achieve this,  we first observe that system \eqref{sys2} can be equivalently reformulated as a single integral equation. In fact, by the linearity of the single-layer potential $V_q[\phi,\cdot]$, we can rewrite the first equation in \eqref{sys2} as
\begin{equation}\label{eq:vrhof}
V_q[\phi,\rho^+ -\rho^-] = f.
\end{equation}
Then, by leveraging the invertibility of the single-layer potential (cf.~\cite[Thm.~2]{LuMu18}) and using equality \eqref{eq:vrhof}, we can express either $\rho^+$ or $\rho^-$ in terms of the other. Substituting this expression into the second equation of \eqref{sys2}, we arrive at the following proposition:
\begin{proposition}\label{prop int eq}
Let $\alpha\in(0,1)$ and $T>0$. Let $\Omega$ be as in \eqref{cond Omega}. Take  $\phi \in {\mathcal{A}}_{\partial\Omega,Q}^{1,\alpha}$.
Assume  $\lambda^+,\lambda^- > 0$ and take $f \in C^{\frac{1+\alpha}{2};1+\alpha}_0([0,T]\times\partial \Omega)$ and 
$g \in C^{\frac{\alpha}{2};\alpha}_0([0,T]\times \partial\Omega)$. Define the contrast {transmission} parameter $\lambda_\mathrm{\bf c}[\lambda^+,\lambda^-]$ by
\begin{equation}\label{eq def lambda}
    \lambda_\mathrm{\bf c}[\lambda^+,\lambda^-] := \frac{\lambda^- - \lambda^+}{\lambda^- + \lambda^+}\,.
\end{equation}
If $(\rho^+,\rho^-) \in C_0^{\frac{\alpha}{2};\alpha}([0,T] \times \partial\Omega) \times C_0^{\frac{\alpha}{2};\alpha}([0,T] \times \partial\Omega)$ is the unique solution of the system of integral equations \eqref{sys2}, then $\rho^-$ is the unique solution in $C_0^{\frac{\alpha}{2};\alpha}([0,T] \times \partial\Omega)$ of the integral equation
\begin{equation}\label{int eq rho-}
    \rho^- - 2\lambda_\mathrm{\bf c}[\lambda^+,\lambda^-] W_{q}^*[\phi,\rho^-] = -\frac{2}{\lambda^- + \lambda^+} \left( \lambda^+ \left( \frac{1}{2}I + W_{q}^*[\phi,\cdot]\right)\left( V_q[\phi,\cdot]^{(-1)}(f)\right) + g \right)
\end{equation}
and $\rho^+$ is given by 
\begin{equation}\label{eq rho+}
    \rho^+ = \rho^- + V_q[\phi,\cdot]^{(-1)}(f).
\end{equation}
\end{proposition}

\begin{proof}
As already noted, equation \eqref{eq rho+} follows by the first equation of \eqref{sys2} and by the linearity and invertibility of the operator $V_q[\phi,\cdot]$ from $C^{\frac{\alpha}{2};\alpha}_0([0,T]\times\partial\Omega)$ to $C^{\frac{1+\alpha}{2};1+\alpha}_0([0,T]\times\partial\Omega)$ (cf.~\cite[Thm.~2]{LuMu18}). Then, substituting \eqref{eq rho+} into the second equation in \eqref{sys2} and using the linearity of $W_{q}^*[\phi,\cdot]$, we obtain 
\begin{align*}
    \lambda^-\left(-\frac{1}{2}\rho^- + W_{q}^*[\phi,\rho^-]\right) &- \lambda^+ \left( \frac{1}{2}\rho^- + \frac{1}{2} V_q[\phi,\cdot]^{(-1)}(f)  \right)
    \\
    &- \lambda^+ \left( W_{q}^*[\phi,\rho^-] + W_{q}^*\left[\phi,V_q[\phi,\cdot]^{(-1)}(f)\right]\right) = g,
\end{align*}
which, after a rearrangement, yields
\begin{align*}
    (\lambda^- + \lambda^+) \left(- \frac{1}{2}\rho^-\right) &+ (\lambda^- - \lambda^+)W_{q}^*[\phi,\rho^-] 
    \\
    &= \lambda^+ \left(\frac{1}{2}I + W_{q}^*[\phi,\cdot]\right)\left( V_q[\phi,\cdot]^{(-1)}(f)\right) + g.
\end{align*}
Multiplying both sides of the above equation by $-\frac{2}{\lambda^- + \lambda^+}$, we obtain \eqref{int eq rho-}, which, in view of \cite[Lem. 2]{LuMu18}, is well known to have a unique solution (cf. the definition of $\lambda_\mathrm{\bf c}[\lambda^+,\lambda^-]$ in \eqref{eq def lambda}).
\end{proof}

In the proof of Proposition \ref{prop int eq}, we utilized the invertibility of the operator $I - 2\gamma W_{q}^*[\phi,\cdot]$ for $\gamma\in (-1,1)$, a fact established in \cite[Lem. 2]{LuMu18}. Even for $\gamma=1$, this operator remains invertible, as follows from \cite[Lem.~6]{Lu18}. In the subsequent lemma, we demonstrate the invertibility of this operator for $\gamma=-1$ as well, thereby establishing its invertibility for all $\gamma\in[-1,1]$.

\begin{lemma}\label{lem: intervibility of K_gamma} 
    Let $\alpha\in(0,1)$ and $T>0$. Let $\Omega$ be as in \eqref{cond Omega}. Let  $\phi \in {\mathcal{A}}_{\partial\Omega,Q}^{1,\alpha}$ and $\gamma \in [-1,1]$. Then the operator from $C_0^{\frac{\alpha}{2};\alpha}([0,T] \times \partial\Omega)$ into itself that maps $\rho$ to the function $\rho -2\gamma W_{q}^*[\phi,\rho]$  
    is a linear homeomorphism.
\end{lemma}

\begin{proof}
As previously noted, the assertion for $\gamma\in (-1,1)$ and $\gamma=1$ follows by \cite[Lem. 2]{LuMu18} and \cite[Lem. 6]{Lu18}, respectively (note that for $\gamma \in (-1,1)$, there exist $\gamma^+,\gamma^->0$ such that $\gamma = (\gamma^- - \gamma^+)/(\gamma^- + \gamma^+)$). Thus, the task at hand is to demonstrate the statement for $\gamma=-1$.

Due to the compactness of $W^*_q[\phi,\cdot]$ (cf.~\cite[Thm. 1]{LuMu18}), the operator $I -2\gamma W_{q}^*[\phi,\cdot]$ is a  Fredholm operator of index zero. Consequently,  to demonstrate that it is a linear homeomorphism, it suffices to prove its injectivity. So, let $\rho \in C_0^{\frac{\alpha}{2};\alpha}([0,T] \times \partial\Omega)$ be such that
    \[
    \rho + 2W_{q}^*[\phi,\rho] = 0 \quad \text{on } [0,T] \times \partial\Omega.
    \]
    By Theorem \ref{thmsl}, the single-layer potential $v_q^+[\rho \circ (\phi^T)^{(-1)}]$ belongs to $C_{0,q}^{\frac{1+\alpha}{2}; 1+\alpha}\big([0,T] \times \overline{\mathbb{S}[\phi]}\big)$ and is a solution of the following $q$-periodic homogeneous interior Neumann problem:
 \begin{equation}\label{q periodic neumann problem: u}
    \begin{cases}
    \partial_t u - \Delta u = 0 &\mbox{in }  (0, T]\times \mathbb{S}[\phi], \\
    u(t,x+qz) = u(t,x) &\forall\,(t,x) \in  [0,T] \times \overline{  \mathbb{S}[\phi]},\, \forall\,z\in \mathbb{Z}^n,  \\
    \frac{\partial} {\partial \nu_\Omega} u = 0 &\mbox{on } [0, T]\times \partial \Omega,\\
    u(0,\cdot) = 0 &\mbox{in } \overline{ \mathbb{S}[\phi]}\, .
    \end{cases}
    \end{equation}
We proceed to prove that $u=0$ is the sole solution of problem \eqref{q periodic neumann problem: u}  by a standard energy argument. It will follow that $v_q^+[\rho \circ (\phi^T)^{(-1)}]=0$ and, by the invertibility  of the restriction to $[0,T]\times \phi(\partial \Omega)$ of the single-layer potential (cf.~\cite[Thm.~2]{LuMu18}), we will conclude that $\rho \circ (\phi^T)^{(-1)}=0$, and thus that $\rho=0$.  

So, let $u \in C_{0,q}^{\frac{1+\alpha}{2};1+\alpha}([0,T]\times\overline{\mathbb{S}[\phi]})$ be a solution of \eqref{q periodic neumann problem: u}. Let 
\begin{equation*}
    e(t) := \int_{\Omega} (u(t,y))^2 \,dy \qquad \forall t \in [0,T].
\end{equation*}
Given that $u$ is uniformly continuous on $[0,T]\times\overline{\mathbb{S}[\phi]}$, we can see that $t\mapsto e(t)$ is continuous on $[0,T]$. Furthermore,  we can demonstrate that $e$ belongs to $C^1([0,T])$.  A detailed proof is provided in \cite[Lem. 5 and Prop. 2]{Lu18},  and it is based on classical differentiation theorems for integrals depending on a parameter, along with a specific approximation of the support of integration (see Verchota \cite[Thm. 1.12, p. 581]{Ve84}).  Following the argument in the same reference (\cite[Lem. 5 and Prop. 2]{Lu18}), we can also verify that
\begin{equation*}
    \frac{d}{dt} e(t) = -2 \int_{\Omega} |Du(t,y)|^2 \,dy + 2 \int_{\partial\Omega} u(t,y) \frac{\partial} {\partial \nu_\Omega} u(t,y) \,d\sigma_y = -2 \int_{\Omega} |Du(t,y)|^2 \,dy \quad \forall t \in (0,T),
\end{equation*}
where the integral on $\partial \Omega$ vanishes thanks to the boundary condition in \eqref{q periodic neumann problem: u}. Hence $\frac{d}{dt} e \leq 0$ in $(0,T)$. Since $e \geq 0$ and $e(0)=0$, we conclude that $e(t) = 0$ for all $t \in [0,T]$. Accordingly,
$u = 0$ on $[0, T] \times \overline{\Omega}$, and the $q$-periodicity of $u$ implies $u=0$ on $[0,T]\times\overline{\mathbb{S}[\phi]}$. Hence 
\[
v_q^+[\rho \circ (\phi^T)^{(-1)}]=0\qquad \mbox{in }  [0, T]\times \overline{\mathbb{S}[\phi]}\,,
\]
a fact that, as explained above, concludes the proof of the statement.
\end{proof}

Taking inspiration from Proposition \ref{prop int eq} and Lemma \ref{lem: intervibility of K_gamma}, we define the map
\begin{equation*}
    \Lambda : {\mathcal{A}}_{\partial\Omega,Q}^{1,\alpha} \times (0,+\infty)^2 \times C^{\frac{1+\alpha}{2};1+\alpha}_0([0,T]\times\partial \Omega) \times C_0^{\frac{\alpha}{2};\alpha}([0,T] \times \partial\Omega) \to C_0^{\frac{\alpha}{2};\alpha}([0,T] \times \partial\Omega)
\end{equation*}
given by
\[
\begin{aligned}
    \Lambda[\phi, \lambda^+,\lambda^-, f,g] :=& \left(I -2     \lambda_\mathrm{\bf c}[\lambda^+,\lambda^-] W_{q}^*[\phi,\cdot]\right)^{(-1)} 
    \\ &
    \left(-\frac{2}{\lambda^- + \lambda^+} \left( \lambda^+ \left( \frac{1}{2}I + W_{q}^*[\phi,\cdot]\right)\left( V_q[\phi,\cdot]^{(-1)}(f)\right) + g \right) \right)\, ,
\end{aligned}
\]
with $\lambda_\mathrm{\bf c}[\lambda^+,\lambda^-]$ as in \eqref{eq def lambda}. Then the solution $\rho^-[\phi,\lambda^+,\lambda^-,f,g]$ to the integral equation in \eqref{int eq rho-} is given by 
\begin{equation}\label{sys3}
\rho^-[\phi,\lambda^+,\lambda^-,f,g] = \Lambda[\phi, \lambda^+,\lambda^-, f,g],
\end{equation}
and if we take 
\begin{equation}\label{sys4}
    \rho^+[\phi,\lambda^+,\lambda^-,f,g] = \rho^-[\phi,\lambda^+,\lambda^-,f,g] + V_q[\phi,\cdot]^{(-1)}(f),
\end{equation}
we see, by  Proposition \ref{prop int eq}, that the pair
\[
\left(\rho^+[\phi,\lambda^+,\lambda^-,f,g],\rho^-[\phi,\lambda^+,\lambda^-,f,g]\right)
\]
is the unique solution of \eqref{sys2}.

Our next objective is to establish a regularity result for the map that takes $(\phi,\lambda^+,\lambda^-,\!f,\!g)$  to $\left(\rho^+[\phi,\lambda^+,\lambda^-,\!f,\!g],\rho^-[\phi,\lambda^+,\lambda^-,\!f,\!g]\right)$, which stems from the smooth dependence of layer potentials on perturbations in the integration's support of Theorem \ref{thm:main}, coupled with the analyticity of the inversion map in Banach algebras. Subsequently, the regularity of the mapping $(\phi,\lambda^+,\lambda^-,\!f,\!g)\mapsto \left(\rho^+[\phi,\lambda^+,\lambda^-,\!f,\!g],\rho^-[\phi,\lambda^+,\lambda^-,\!f,\!g]\right)$ will resolve into a regularity result for the mapping that relates $(\phi,\lambda^+,\lambda^-,\!f,\!g)$ with the solution of \eqref{periodicidtran}.

\begin{proposition}\label{prop:aninteq}
Let $\alpha\in(0,1)$ and $T>0$. Let $\Omega$ be as in \eqref{cond Omega}.
Then the map
\[
(\phi,\lambda^+,\lambda^-,\!f,\!g)\mapsto \left(\rho^+[\phi,\lambda^+,\lambda^-,\!f,\!g],\rho^-[\phi,\lambda^+,\lambda^-,\!f,\!g]\right)
\] is
 of class $C^\infty$ from ${\mathcal{A}}_{\partial\Omega,Q}^{1,\alpha} \times (0,+\infty)^2 \times
 C^{\frac{1+\alpha}{2};1+\alpha}_0([0,T]\times\partial \Omega) \times
  C^{\frac{\alpha}{2};\alpha}_0([0,T]\times\partial \Omega)$ to 
 $C_0^{\frac{\alpha}{2};\alpha}([0,T] \times \partial\Omega)\times C_0^{\frac{\alpha}{2};\alpha}([0,T] \times \partial\Omega)$.	
\end{proposition}

\begin{proof}
    By Theorem \ref{thm:main}, the map that takes $\phi$ to $V_q[\phi,\cdot]$ is of class $C^\infty$ from ${\mathcal{A}}_{\partial\Omega,Q}^{1,\alpha}$ to $\mathcal{L} \left(C^{\frac{\alpha}{2};\alpha}_0([0,T]\times\partial\Omega), C^{\frac{1+\alpha}{2};1+\alpha}_0([0,T]\times\partial\Omega)\right)$, and the map that takes $(\phi, \gamma)$ to $I -2 \gamma W_{q}^*[\phi,\cdot]$ is of class $C^\infty$ from $ {\mathcal{A}}_{\partial\Omega,Q}^{1,\alpha}\times (-1,1)$ to $\mathcal{L}\left( C^{\frac{\alpha}{2};\alpha}_0([0,T]\times\partial\Omega), C^{\frac{\alpha}{2};\alpha}_0([0,T]\times\partial\Omega)\right)$. Since the map from $(0,+\infty)^2$ to $(-1,1)$ that takes $(\lambda^+,\lambda^-)$ to $\frac{\lambda^- - \lambda^+}{\lambda^- + \lambda^+}$ is also of class $C^\infty$, we deduce that the map from ${\mathcal{A}}_{\partial\Omega,Q}^{1,\alpha}\times(0,+\infty)^2$ to $\mathcal{L}\left( C^{\frac{\alpha}{2};\alpha}_0([0,T]\times\partial\Omega), C^{\frac{\alpha}{2};\alpha}_0([0,T]\times\partial\Omega)\right)$ that takes a triple $(\phi,\lambda^+,\lambda^-)$ to
    \[I -2 \frac{\lambda^- - \lambda^+}{\lambda^- + \lambda^+} W_{q}^*[\phi,\cdot]\] is of class $C^\infty$.  
    
Now, the map that takes a linear invertible operator to its inverse is real analytic (cf.~Hille and Phillips \cite[Thms. 4.3.2 and 4.3.4]{HiPh57}), and therefore of class $C^\infty$. So, by the invertibility of the periodic single layer of \cite[Thm.~2]{LuMu18} and by Lemma \ref{lem: intervibility of K_gamma}  we deduce that the map from ${\mathcal{A}}_{\partial\Omega,Q}^{1,\alpha}$ to $\mathcal{L} \left(C^{\frac{1+\alpha}{2};1+\alpha}_0([0,T]\times\partial\Omega), C^{\frac{\alpha}{2};\alpha}_0([0,T]\times\partial\Omega)\right)$ that takes $\phi$ to $V_q[\phi,\cdot]^{(-1)}$ and the map from ${\mathcal{A}}_{\partial\Omega,Q}^{1,\alpha} \times (0,+\infty)^2$ to $\mathcal{L} \left(C^{\frac{\alpha}{2};\alpha}_0([0,T]\times\partial\Omega), C^{\frac{\alpha}{2};\alpha}_0([0,T]\times\partial\Omega)\right)$
  that takes $(\phi,\lambda^+,\lambda^-)$ to $\left(I -2 \frac{\lambda^- - \lambda^+}{\lambda^- + \lambda^+} W_{q}^*[\phi,\cdot]\right)^{(-1)}$, are both of class $C^\infty$.

Given the bilinearity and continuity of the evaluation map $(L,v)\mapsto L[v]$, which acts from 
\[
\mathcal{L} \left(C^{\frac{1+\alpha}{2};1+\alpha}_0([0,T]\times\partial\Omega), C^{\frac{\alpha}{2};\alpha}_0([0,T]\times\partial\Omega)\right)\times C^{\frac{1+\alpha}{2};1+\alpha}_0([0,T]\times\partial\Omega)
\]
to $C_0^{\frac{\alpha}{2};\alpha}([0,T] \times \partial\Omega)$, as well as from
\[
\mathcal{L}\left( C^{\frac{\alpha}{2};\alpha}_0([0,T]\times\partial\Omega), C^{\frac{\alpha}{2};\alpha}_0([0,T]\times\partial\Omega)\right)\times C^{\frac{\alpha}{2};\alpha}_0([0,T]\times\partial\Omega)
\]
to $C^{\frac{\alpha}{2};\alpha}_0([0,T]\times\partial\Omega)$, we can deduce that the mapping $(\phi,f)\mapsto V_q[\phi,\cdot]^{(-1)}(f)$ is of class $C^\infty$ from ${\mathcal{A}}_{\partial\Omega,Q}^{1,\alpha} \times C^{\frac{1+\alpha}{2};1+\alpha}_0([0,T]\times\partial \Omega)$ to $C_0^{\frac{\alpha}{2};\alpha}([0,T] \times \partial\Omega)$ and, similarly, the map
\[
(\phi,\lambda^+,f)\mapsto\lambda^+ \left( \frac{1}{2}I + W_{q}^*[\phi,\cdot]\right)\left( V_q[\phi,\cdot]^{(-1)}(f)\right)
\]
is of class $C^\infty$ from ${\mathcal{A}}_{\partial\Omega,Q}^{1,\alpha}\times (0,+\infty) \times  C^{\frac{\alpha}{2};\alpha}_0([0,T]\times\partial\Omega)$ to $C^{\frac{\alpha}{2};\alpha}_0([0,T]\times\partial\Omega)$.

By once again relying on the bilinearity and continuity of the evaluation map, we ultimately deduce that the map taking $(\phi,\lambda^+,\lambda^-,f,g)$ to
\[
\left(I -2 \frac{\lambda^- - \lambda^+}{\lambda^- + \lambda^+} W_{q}^*[\phi,\cdot]\right)^{(-1)}
\left(-\frac{2}{\lambda^- + \lambda^+} \left( \lambda^+ \left( \frac{1}{2}I + W_{q}^*[\phi,\cdot]\right)\left( V_q[\phi,\cdot]^{(-1)}(f)\right) + g \right) \right)
\]
is of class $C^\infty$, where the domain is ${\mathcal{A}}_{\partial\Omega,Q}^{1,\alpha} \times (0,+\infty)^2\times 
C^{\frac{1+\alpha}{2};1+\alpha}_0([0,T]\times\partial \Omega) \times
C^{\frac{\alpha}{2};\alpha}_0([0,T]\times\partial \Omega)$, and the codomain is $C_0^{\frac{\alpha}{2};\alpha}([0,T] \times \partial\Omega)$.

Hence, the smoothness of the map $(\phi,\lambda^+,\lambda^-,f,g)\mapsto \rho^-[\phi,\lambda^+,\lambda^-,f,g]$ follows directly from \eqref{sys3} and the definition of $\Lambda$. The smoothness of  $(\phi,\lambda^+,\lambda^-,f,g)\mapsto \rho^+[\phi,\lambda^+,\lambda^-,f,g]$ is a consequence of  \eqref{sys4}.
\end{proof}

Theorem \ref{periodiccp} provides a representation formula for the solution of problem \eqref{periodicidtran} in terms of periodic single-layer potentials, while Proposition \ref{prop:aninteq} demonstrates that the corresponding densities exhibit smooth dependence on the shape, boundary data, and {transmission} parameters. Specifically, we have the expressions
\begin{equation}\label{repfu1}
\begin{split}
    u^+[\phi,\lambda^+,\lambda^-,f,g]&(t,x)
    \\
    = \int_0^{t} \int_{\partial\Omega}&S_{q,n}(t-\tau, x - \phi(y)) \rho^+[\phi,\lambda^+,\lambda^-,f,g](\tau,y)
    \tilde \sigma_n[\phi](y)\, d\sigma_yd\tau,
\end{split}
\end{equation}
for all $(t,x) \in [0,T] \times  {\overline{\mathbb{S}[\phi]}}$, and 
\begin{equation}\label{repfu2}
\begin{split}
    u^-[\phi,\lambda^+,\lambda^-,f,g]&(t,x) 
    \\
    = \int_0^{t} \int_{\partial\Omega}&S_{q,n}(t-\tau, x - \phi(y)) \rho^-[\phi,\lambda^+,\lambda^-,f,g](\tau,y) \tilde \sigma_n[\phi](y)\, d\sigma_yd\tau,
\end{split}
\end{equation}
for all $(t,x) \in [0,T] \times   {\overline{\mathbb{S}[\phi]^-}}$,  where   $\rho^+[\phi,\lambda^+,\lambda^-,f,g]$ and $\rho^-[\phi,\lambda^+,\lambda^-,f,g]$ are maps of class $C^\infty$ with respect to the variables $(\phi,\lambda^+,\lambda^-,f,g$). We are ready to show the main result of this section, about the smooth dependence of the solution of \eqref{periodicidtran} on $(\phi,\lambda^+,\lambda^-,f,g)$.

%\todo[size=\tiny]{ PM: da qui in poi si presenta il problema che abbiamo a che fare con prodotti di (chiusure di aperti -eventualmente regolari-) e frontiere, mentre alcuni risultati sono direttamente per manifold. E' lo stesso problema che avevamo anni fa utilizzando Valent prima di avere il risultato sugli  operatori integrali con kernel analitico.}
\begin{theorem}
Let $\alpha\in(0,1)$ and $T>0$. Let $\Omega$ be as in \eqref{cond Omega}. Let $\Omega^i$ and $\Omega^e$ be two bounded open subsets of $\mathbb{R}^n$. Let ${\mathcal{B}}_{\partial\Omega,Q}^{1,\alpha}$
be the open subset of ${\mathcal{A}}_{\partial\Omega,Q}^{1,\alpha}$ consisting of those diffeomorphisms $\phi$
such that
\[
\overline{\Omega^i} \subseteq \mathbb{S}[\phi], \quad \overline{\Omega^e} \subseteq \mathbb{S}[\phi]^-.
\]
 Then, the map 
 \[
 (\phi,\lambda^+,\lambda^-,f,g) \mapsto \left(u^+[\phi,\lambda^+,\lambda^-,f,g]_{|[0,T]\times\overline{\Omega^i}},u^-[\phi,\lambda^+,\lambda^-,f,g]_{|[0,T]\times\overline{\Omega^e}}\right)
 \]
 is of class $C^\infty$ from ${\mathcal{B}}_{\partial\Omega,Q}^{1,\alpha} \times (0,+\infty)^2 \times C^{\frac{1+\alpha}{2};1+\alpha}_0([0,T]\times\partial \Omega) \times
  C^{\frac{\alpha}{2};\alpha}_0([0,T]\times\partial \Omega)$ to  $C_0^{\frac{1+\alpha}{2},1+\alpha}([0,T]\times \overline{\Omega^i})
  \times C_0^{\frac{1+\alpha}{2},1+\alpha}([0,T]\times \overline{\Omega^e})$.%\todo[size=\tiny]{PM: nello spazio siamo all'interno; riusciamo a fare di pi\`u come regolarit\`a?}
\end{theorem}
\begin{proof}
%\todo[size=\tiny]{PM: Sistemare la dimostrazione!} 
Without loss of generality we can assume that $\Omega^i$ and $\Omega^e$ are of class $C^{1,\alpha}$. The maps that associate a diffeomorphism $\phi$ with the functions
\[
\overline{\Omega^i} \times \partial \Omega \ni (x,y) \mapsto  x -\phi(y)\in \mathbb{R}^n
\]
and
\[
\overline{\Omega^e} \times \partial \Omega \ni (x,y) \mapsto  x -\phi(y)\in \mathbb{R}^n
\]
are both affine and continuous (and thus, smooth), from ${\mathcal{B}}_{\partial\Omega,Q}^{1,\alpha}$ to $C^{1,\alpha}(\overline{\Omega^i} \times \partial \Omega, \mathbb{R}^n \setminus q\mathbb{Z}^n)$ and $C^{1,\alpha}(\overline{\Omega^e} \times \partial \Omega, \mathbb{R}^n \setminus q\mathbb{Z}^n)$, respectively. %\todo[size=\tiny]{THEY ARE NOT LINEAR!} 
 By arguing as in the proof of \cite[Lem. A.1 and Lem. A.3]{DaLu23} regarding the regularity of superposition operators, we deduce that the maps that take $\phi$ to the functions 
\[
    S_{q,n}(t, x-\phi(y))\qquad\forall[0,T]\times\overline{\Omega^i} \times \partial \Omega
\]
and 
\[
    S_{q,n}(t, x-\phi(y))\qquad\forall[0,T]\times\overline{\Omega^e} \times \partial \Omega
\]
are of class $C^{\infty}$ from ${\mathcal{B}}_{\partial\Omega,Q}^{1,\alpha}$ to $C_0^{\frac{1+\alpha}{2};1+\alpha}([0,T]\times (\overline{\Omega^i} \times \partial \Omega))$ and to $C_0^{\frac{1+\alpha}{2};1+\alpha}([0,T]\times (\overline{\Omega^e} \times \partial \Omega))$, respectively. Indeed, we note that the results of \cite[Lem. A.1 and Lem. A.3]{DaLu23} remain valid also in the case of a manifold with a boundary.

Then, the statement follows by the representation formulas  \eqref{repfu1}, \eqref{repfu2} for $u^\pm[\phi,\lambda^+,\lambda^-,f,g]$, by Proposition \ref{prop:aninteq} on the smoothness of $\rho^\pm[\phi,\lambda^+,\lambda^-,f,g]$, by Lemma \ref{rajacon} on the analyticity of $\tilde \sigma_n[\phi]$, and by the regularity result on integral operators with non-singular kernels of   \cite[Lem.~A.2]{DaLu23}, which continues to apply even in the case of a manifold with a boundary.
\end{proof}

\section{An expansion result by Neumann-type series} \label{s:neu}

If we consider fixed values of $\phi \in {\mathcal{A}}_{\partial\Omega,Q}^{1,\alpha}$, $f \in C^{\frac{1+\alpha}{2};1+\alpha}_0([0,T]\times\partial \Omega)$, and $g \in C^{\frac{\alpha}{2};\alpha}_0([0,T]\times \Omega)$, a combination of Proposition \ref{prop int eq} and a modified version of Proposition \ref{prop:aninteq} allows us to establish that the solution to problem \eqref{periodicidtran} exhibits analytic dependence on the {term} $\lambda_\mathrm{\bf c}[\lambda^+,\lambda^-]$. Consequently, we can express the densities as convergent power series. Alternatively, this result can be achieved more directly by employing the Neumann series Theorem.

To be more precise, we can demonstrate that locally, around a fixed pair of parameters $(\lambda^+_0,\lambda^-_0) \in (0,+\infty)^2$, the densities can be expressed {by means of} a Neumann-type series. The terms of this series involve the difference of the {terms} $\lambda_\mathrm{\bf c}[\lambda^+,\lambda^-]$ and $\lambda_\mathrm{\bf c}[\lambda^+_0,\lambda^-_0]$,  as well as iterated compositions of the operator
\begin{equation*}
\left(I - 2 \lambda_\mathrm{\bf c}[\lambda^+_0,\lambda^-_0] W^*_q[\phi,\cdot] \right)^{(-1)} \circ W^\ast_q[\phi,\cdot].   
\end{equation*}
Naturally, once we establish this result for the densities, by utilizing the representation formula of the solution in terms of space-periodic layer potentials,  we can deduce a similar result for the solution. The detailed calculation is left to the zealous reader.

We will use the following notation: Given two Banach spaces $X$ and $Y$ and a bounded linear map $T:X \to Y$, we define
\begin{equation*}
    T^j := \underbrace{ T \circ \dots \circ T}_{j-\text{times}} \quad \text{for every } j \in \mathbb{N},
\end{equation*}
with the convention that $T^0= I$.

In the theorem below, we fix $\phi\in {\mathcal{A}}_{\partial\Omega,Q}^{1,\alpha}$, $\lambda^+_0,\lambda^-_0 > 0$,  $f \in C^{\frac{1+\alpha}{2};1+\alpha}_0([0,T]\times\partial \Omega)$, and $g \in C^{\frac{\alpha}{2};\alpha}_0([0,T]\times \Omega)$ and we show a representation formula for $\rho^-[\phi,\lambda^+,\lambda^-,f,g]$ as a convergent power series depending on the difference of the {terms} $\lambda_\mathrm{\bf c}[\lambda^+,\lambda^-]$ and $\lambda_\mathrm{\bf c}[\lambda^+_0,\lambda^-_0]$. For the sake of exposition, for every $j \in {\mathbb{N}}$, we define the map
\begin{equation*}
   \mathcal{K}_j: {\mathcal{A}}_{\partial\Omega,Q}^{1,\alpha} \times (0,+\infty)^2 \to \mathcal{L}(C_0^{\frac{\alpha}{2};\alpha}([0,T] \times \partial\Omega),C_0^{\frac{\alpha}{2};\alpha}([0,T] \times \partial\Omega)),
\end{equation*}
given by 
\begin{equation}\label{eq K_j}
\mathcal{K}_j[\phi,\lambda^+_0,\lambda^-_0] := 2^j \left(\left(I - 2 \lambda_\mathrm{\bf c}[\lambda^+_0,\lambda^-_0] W^*_q[\phi,\cdot] \right)^{(-1)} {\circ}W^\ast_q[\phi,\cdot] \right)^j.
\end{equation}
Then the following holds.

\begin{theorem}\label{thm:neumser}
Let $\alpha\in(0,1)$ and $T>0$. Let $\Omega$ be as in \eqref{cond Omega}. Let  $\phi \in {\mathcal{A}}_{\partial\Omega,Q}^{1,\alpha}$, $\lambda^+_0,\lambda^-_0 > 0$, $f \in C^{\frac{1+\alpha}{2};1+\alpha}_0([0,T]\times\partial \Omega)$, and $g \in C^{\frac{\alpha}{2};\alpha}_0([0,T]\times \Omega)$ be fixed.

Then, there exists a positive constant $\varepsilon \in (0,+\infty)$ such that the following holds: For every $(\lambda^+,\lambda^-) \in (0,+\infty)^2$ such that
\begin{equation}\label{thm condition epsilon}
    |\lambda_\mathrm{\bf c}[\lambda^+,\lambda^-] - \lambda_\mathrm{\bf c}[\lambda^+_0,\lambda^-_0]| <
 \varepsilon,
\end{equation}
with $\lambda_\mathrm{\bf c}[\cdot,\cdot]$ defined by \eqref{eq def lambda}, we have
\begin{equation}\label{thm eq rho- series}
        \begin{split}
            \rho^-[\phi,\lambda^+,\lambda^-,f,g] =& \left( \sum_{j=0}^{+\infty} (\lambda_\mathrm{\bf c}[\lambda^+,\lambda^-] - \lambda_\mathrm{\bf c}[\lambda^+_0,\lambda^-_0])^j \mathcal{K}_j[\phi,\lambda^+_0,\lambda^-_0]  \right)
            \\
            &
            \circ \left(I - 2 \lambda_\mathrm{\bf c}[\lambda^+_0,\lambda^-_0] W^*_q[\phi,\cdot] \right)^{(-1)} (\rho^-_0[\phi,\lambda^+,\lambda^-,f,g]), 
        \end{split}
    \end{equation}
where the series 
\[
\sum_{j=0}^{+\infty} \zeta^j \mathcal{K}_j[\phi,\lambda^+_0,\lambda^-_0]
\]
converges normally in $\mathcal{L}(C^{\frac{\alpha}{2};\alpha}_0([0,T]\times \partial\Omega), C^{\frac{\alpha}{2};\alpha}_0([0,T]\times \partial\Omega))$ for $|\zeta|<\varepsilon$ and where
\begin{equation}\label{eq rho-0}
    \rho^-_0[\phi,\lambda^+,\lambda^-,f,g] := -\frac{2}{\lambda^- + \lambda^+} \left( \lambda^+ \left( \frac{1}{2}I + W_{q}^*[\phi,\cdot]\right)\left( V_q[\phi,\cdot]^{(-1)}(f)\right) + g \right).
\end{equation}

\end{theorem}

\begin{proof}
    Let $\phi \in {\mathcal{A}}_{\partial\Omega,Q}^{1,\alpha}$, $\lambda^+_0,\lambda^-_0 > 0$,  $f \in C^{\frac{1+\alpha}{2};1+\alpha}_0([0,T]\times\partial \Omega)$, and $g \in C^{\frac{\alpha}{2};\alpha}_0([0,T]\times \partial\Omega)$. We first notice that, by the definition of $\rho^-_0$ in \eqref{eq rho-0}, we can rewrite \eqref{int eq rho-} as
    \begin{equation}\label{eq I -2lambda W* rho- = rho-0}
        \left(I - 2 \lambda_\mathrm{\bf c}[\lambda^+,\lambda^-] W^*_q[\phi,\cdot] \right) \rho^-[\phi,\lambda^+,\lambda^-,f,g] = \rho^-_0[\phi,\lambda^+,\lambda^-,f,g] \, ,
    \end{equation}
    for every $(\lambda^+,\lambda^-) \in (0,+\infty)^2$. We now consider the operator on the left-hand side of \eqref{eq I -2lambda W* rho- = rho-0}, which is $I - 2 \lambda_\mathrm{\bf c}[\lambda^+,\lambda^-] W^*_q[\phi,\cdot]: C^{\frac{\alpha}{2};\alpha}_0([0,T]\times \partial\Omega) \to C^{\frac{\alpha}{2};\alpha}_0([0,T]\times \partial\Omega)$. By adding and subtracting the term $2 \lambda_\mathrm{\bf c}[\lambda^+_0,\lambda^-_0] W^*_q[\phi,\cdot]$ and factoring out the operator $I - 2 \lambda_\mathrm{\bf c}[\lambda^+_0,\lambda^-_0] W^*_q[\phi,\cdot]$, we can rewrite this operator as follows:    \begin{equation}\label{identity for I-2lambda_c W^*}
    \begin{split}
        I - 2& \lambda_\mathrm{\bf c}[\lambda^+,\lambda^-] W^*_q[\phi,\cdot] 
        \\
        =& I - 2 \lambda_\mathrm{\bf c}[\lambda^+_0,\lambda^-_0] W^*_q[\phi,\cdot] - 2 (\lambda_\mathrm{\bf c}[\lambda^+,\lambda^-] - \lambda_\mathrm{\bf c}[\lambda^+_0,\lambda^-_0]) W^*_q[\phi,\cdot]
        \\
        =&\left( I - 2 \lambda_\mathrm{\bf c}[\lambda^+_0,\lambda^-_0] W^*_q[\phi,\cdot]\right)
        \\
        &\circ \left(I - 2 (\lambda_\mathrm{\bf c}[\lambda^+,\lambda^-] - \lambda_\mathrm{\bf c}[\lambda^+_0,\lambda^-_0]) \left(I - 2 \lambda_\mathrm{\bf c}[\lambda^+_0,\lambda^-_0] W^*_q[\phi,\cdot]\right)^{(-1)} {\circ}W^*_q[\phi,\cdot]\right).
    \end{split}
    \end{equation}
    In particular, by \eqref{identity for I-2lambda_c W^*},
    we deduce that \begin{equation}\label{identity for (I-2lambda_c W^*)^(-1)}
        \begin{split}
            &\left(I - 2 \lambda_\mathrm{\bf c}[\lambda^+,\lambda^-] W^*_q[\phi,\cdot]\right)^{(-1)} 
            \\
            & \qquad=\left(I - 2 (\lambda_\mathrm{\bf c}[\lambda^+,\lambda^-] - \lambda_\mathrm{\bf c}[\lambda^+_0,\lambda^-_0]) \left(I - 2 \lambda_\mathrm{\bf c}[\lambda^+_0,\lambda^-_0] W^*_q[\phi,\cdot]\right)^{(-1)} {\circ} W^*_q[\phi,\cdot]\right)^{(-1)} 
            \\
            &\qquad\quad\circ \left( I - 2 \lambda_\mathrm{\bf c}[\lambda^+_0,\lambda^-_0] W^*_q[\phi,\cdot]\right)^{(-1)}.
        \end{split}
    \end{equation}
    Then, if we choose $\varepsilon>0$ small enough, for example
    \begin{equation}\label{eq varepsilon}
        \varepsilon := \frac{1}{2 \left\| \left(I - 2 \lambda_\mathrm{\bf c}[\lambda^+_0,\lambda^-_0] W^*_q[\phi,\cdot]\right)^{(-1)} {\circ}W^*_q[\phi,\cdot]\right\|_{\mathcal{L}\left( C^{\frac{\alpha}{2};\alpha}_0([0,T]\times\partial\Omega), C^{\frac{\alpha}{2};\alpha}_0([0,T]\times\partial\Omega)\right)}},
    \end{equation}
    we have that, for every $(\lambda^+,\lambda^-) \in (0,+\infty)^2$ such that \eqref{thm condition epsilon} holds, the inverse of the operator
    \begin{equation*}
        I - 2 (\lambda_\mathrm{\bf c}[\lambda^+,\lambda^-] - \lambda_\mathrm{\bf c}[\lambda^+_0,\lambda^-_0]) \left(I - 2 \lambda_\mathrm{\bf c}[\lambda^+_0,\lambda^-_0] W^*_q[\phi,\cdot]\right)^{(-1)}{\circ} W^*_q[\phi,\cdot]
    \end{equation*}
    from $C^{\frac{\alpha}{2};\alpha}_0([0,T]\times \partial\Omega)$ into itself can be written as a normally convergent Neumann series in $\mathcal{L}\left(C^{\frac{\alpha}{2};\alpha}_0([0,T]\times \partial\Omega), C^{\frac{\alpha}{2};\alpha}_0([0,T]\times\partial\Omega)\right)$. In fact, by \eqref{thm condition epsilon} and \eqref{eq varepsilon}, and by the {Neumann series Theorem}, we have that
    \begin{equation}\label{eq neumann series operator}
        \begin{split}
            \left(I - 2 (\lambda_\mathrm{\bf c}[\lambda^+,\lambda^-] - \lambda_\mathrm{\bf c}[\lambda^+_0,\lambda^-_0]) \left(I - 2 \lambda_\mathrm{\bf c}[\lambda^+_0,\lambda^-_0] W^*_q[\phi,\cdot]\right)^{(-1)} {\circ}W^*_q[\phi,\cdot]\right)^{(-1)}  
            \\
            = \sum_{j=0}^{+\infty} (\lambda_\mathrm{\bf c}[\lambda^+,\lambda^-] - \lambda_\mathrm{\bf c}[\lambda^+_0,\lambda^-_0])^j \mathcal{K}_j[\phi,\lambda^+_0,\lambda^-_0]\, ,   
        \end{split}
    \end{equation}
    where for each $j \in \mathbb{N}$ the operator $\mathcal{K}_j[\cdot,\cdot,\cdot]$ is defined by \eqref{eq K_j}. Finally, \eqref{eq I -2lambda W* rho- = rho-0}, \eqref{identity for (I-2lambda_c W^*)^(-1)} and \eqref{eq neumann series operator} yield to the validity of \eqref{thm eq rho- series}.
\end{proof}

{\begin{remark}
Let the assumptions of Theorem \ref{thm:neumser} hold. By equations \eqref{thm eq rho- series} and {\eqref{eq rho-0}}, we have
   \[
   \begin{split}
            \rho^-[\phi,\lambda^+,\lambda^-,f,g] =&  -\frac{2 \lambda^+}{\lambda^- + \lambda^+} \left( \sum_{j=0}^{+\infty} (\lambda_\mathrm{\bf c}[\lambda^+,\lambda^-] - \lambda_\mathrm{\bf c}[\lambda^+_0,\lambda^-_0])^j \mathcal{K}_j[\phi,\lambda^+_0,\lambda^-_0]  \right)
            \\
            &
            \circ \left(I - 2 \lambda_\mathrm{\bf c}[\lambda^+_0,\lambda^-_0] W^*_q[\phi,\cdot] \right)^{(-1)}  \left( \frac{1}{2}I + W_{q}^*[\phi,\cdot]\right)\left( V_q[\phi,\cdot]^{(-1)}(f)\right) )\\
            & -\frac{2}{\lambda^- + \lambda^+}\left( \sum_{j=0}^{+\infty} (\lambda_\mathrm{\bf c}[\lambda^+,\lambda^-] - \lambda_\mathrm{\bf c}[\lambda^+_0,\lambda^-_0])^j \mathcal{K}_j[\phi,\lambda^+_0,\lambda^-_0]  \right)
            \\
            &
            \circ \left(I - 2 \lambda_\mathrm{\bf c}[\lambda^+_0,\lambda^-_0] W^*_q[\phi,\cdot] \right)^{(-1)} \left(  g \right)\, ,
        \end{split}
  \]
 for every  $(\lambda^+,\lambda^-) \in (0,+\infty)^2$ such that condition \eqref{thm condition epsilon} holds.
\end{remark}}

\subsection*{Acknowledgment}

The authors are members of the ``Gruppo Nazionale per l'Analisi Matematica, la Probabilit\`a e le loro Applicazioni'' (GNAMPA) of the ``Istituto Nazionale di Alta Matematica'' (INdAM). M.D.R., P.L.~and P.M.~acknowledge the support of the ``INdAM GNAMPA Project'' codice CUP\_E53C22001930001  ``Operatori differenziali e integrali in geometria spettrale'' {and of the project funded by the EuropeanUnion -- NextGenerationEU under the National Recovery and Resilience Plan (NRRP), Mission 4 Component 2 Investment 1.1 - Call PRIN 2022 No. 104 of February 2, 2022 of Italian Ministry of University and Research; Project 2022SENJZ3 (subject area: PE - Physical Sciences and Engineering) ``Perturbation problems and asymptotics for elliptic differential equations: variational and potential theoretic methods''.} P.M.~and R.M.~also acknowledge  the support of  the SPIN Project  ``DOMain perturbation problems and INteractions Of scales - DOMINO''  of the Ca' Foscari University of Venice. P.M.~also acknowledges the support from EU through the H2020-MSCA-RISE-2020 project EffectFact, 
Grant agreement ID: 101008140. Part of this work was done while R.M.~was visiting C3M - Centre for Computational Continuum Mechanics (Slovenia). R.M.~wishes to thank C3M  for the kind hospitality.


\begin{thebibliography}{11}

\bibitem{BeEbSo11}
S.~Bernstein, S.~Ebert, and R.~ S\"oren Krau{\ss}har,  On the diffusion equation and diffusion wavelets on flat cylinders and the $n$-torus, {\it Math. Methods Appl. Sci.} {\bf34} (2011), no. 4, 428--441. 


\bibitem{ChKrYo98}
R. Chapko, R. Kress, and J. R. Yoon, 
On the numerical solution of an inverse boundary 
value problem for the heat equation,
{\it Inverse Problems} {\bf 14}  (1998),  853--867.

%\bibitem{ChKrYo99}
%R. Chapko, R. Kress, and J. R. Yoon,
% An inverse boundary value problem for the heat equation: 
%the Neumann condition,
%{\it Inverse Problems} {\bf 15}  (1999),  1033--1046. 


%\bibitem{Ch95} 
%A.~Charalambopoulos, On the Fr\'echet differentiability of boundary integral operators in the inverse elastic scattering problem, {\it Inverse Problems} {\bf 11} (1995), 1137--1161.
%
 
% \bibitem{CoScZe18} 
%A.~Cohen, C.~Schwab, and J.~Zech, Shape holomorphy of the stationary Navier-Stokes equations, {\it SIAM J. Math. Anal.} {\bf 50} (2018), no. 2, 1720--1752. 

%\bibitem{CoKr13}
% D.~Colton and R.~Kress,  Integral equation methods in scattering theory. Reprint of the 1983 original. Classics in Applied Mathematics, 72. Society for Industrial and Applied Mathematics (SIAM), Philadelphia, PA, 2013.
% 
% \bibitem{CoKr19} D.~Colton and R.~Kress, {Inverse acoustic and electromagnetic scattering theory.} Applied Mathematical Sciences 93, Springer, Cham,  2019.
% 
%\bibitem{Con19}
%J.B.~Conway, {\it A course in functional analysis}. Vol. 96, Springer, 2019.

%\bibitem{CoLe12a} 
%M.~Costabel and F.~Le Lou\"er, Shape derivatives of boundary integral operators in electromagnetic scattering. Part I: Shape differentiability of pseudo-homogeneous boundary integral operators, {\it Integral Equations Oper.~Theory} {\bf 72} (2012),  509--535.

%\bibitem{CoLe12b} 
%M.~Costabel and F.~Le Lou\"er, Shape derivatives of boundary integral operators in electromagnetic scattering. Part II: Application to scattering by a homogeneous dielectric obstacle, {\it Integral Equations 
%Oper.~Theory} {\bf 73} (2012),  17--48.


%
%  \bibitem{BoTo73}
%R. B\"ohme, F. Tomi, Zur Struktur der L\"osungsmenge des Plateauproblems. {\it Math Z.} {\bf 133} (1973), 1--29.
%
%
%\bibitem{BuBu05}
%D. Bucur, G. Buttazzo, Variational methods in shape optimization problems. Progress in Nonlinear Differential Equations and their Applications, 65. Birkh\"auser Boston, Inc., Boston, MA, 2005. viii+216 pp. 
%
%\bibitem{ChKrYo98}
%R. Chapko, R. Kress, and J.R. Yoon, 
%On the numerical solution of an inverse boundary value problem for the heat equation.  {\it Inverse Problems} {\bf 14} (1998), no. 4, 853--867.
%
%\bibitem{ChKrYo99}
%R. Chapko, R. Kress, and J.R. Yoon, An Inverse Boundary Value Problem
%for the Heat Equation the Neumann Condition.  {\it J. Integral Equations Appl.}  {\bf 9} (1997), no. 1, 47--69.
%%
%\bibitem{CoMe83}
%R.R. Coifman and Y. Meyer, Lavrentiev?s curves and conformal mappings, Report No. 5, Institut Mittag-Leffler, 1983.


%\bibitem{Ch99}
% M.F. Cherepova, On some properties of the parabolic potential of bulk masses. I. (Russian) Differ. Uravn. 35 (1999), no. 12, 1701--1706, 1728; translation in Differential Equations 35 (1999), no. 12, 1726--1732 (2000)
%

\bibitem{DaLa10}
M. Dalla Riva and M. Lanza de Cristoforis, A perturbation result for the layer potentials of general second order differential operators with constant coefficients, {\it J. Appl. Funct. Anal.}  {\bf 5} (2010), no. 1, 10--30.


 

\bibitem{DaLaMu21}
M. Dalla Riva, M. Lanza de Cristoforis, and P. Musolino, {\it Singularly Perturbed Boundary Value Problems: A Functional Analytic Approach}, Springer Nature, Cham, 2021.


\bibitem{DaLu23}
M. Dalla Riva and P. Luzzini, Regularity of layer heat potentials upon perturbation of the space support in the optimal H\"older setting{,} {\it Differ. Integral Equ.} {\bf 36} (2023), no. 11-12, 971--1003.

%\bibitem{DaLuMu22}
%M.~Dalla Riva, P.~Luzzini, and P.~Musolino,  Shape analyticity and singular perturbations for layer potential operators, {\it ESAIM Math. Model. Numer. Anal.} {\bf 56} (2022), no. 6, 1889--1910.

\bibitem{DaLuMuPu22}
M. Dalla Riva, P. Luzzini, P. Musolino, and R. Pukhtaievych, Dependence of effective properties upon regular perturbations, In I. Andrianov, S. Gluzman, V. Mityushev, Editors, {\it Mechanics and Physics of Structured Media. Asymptotic and Integral Equations Methods of Leonid Filshtinsky}. Academic Press, Elsevier, London, 2022, pp. 271--301.




%\bibitem{Da96}
%D. Daners,  Domain perturbation for linear and nonlinear parabolic equations. {\it J. Differential Equations} 
% {\bf 129} (1996), no. 2, 358--402.
%
% 
% \bibitem{Da08}
% D. Daners, Domain perturbation for linear and semi-linear boundary value problems, Handbook of differential equations: stationary partial differential equations, Vol. VI, 1-81, Handb. Differ. Equ., Elsevier/North-Holland, Amsterdam, 2008.

%\bibitem{Dav07}
%E.B.~ Davies, {\it Linear operators and their spectra}, Vol. {\bf 106}, Cambridge University Press, 2007.

%\bibitem{De85}
%K.~Deimling, {\it Nonlinear {F}unctional {A}nalysis}. Springer-Verlag, Berlin, 1985. 
%
% \bibitem{DeZo11}  
% M.C.~Delfour, J.P.~Zol\'esio, Shapes and geometries. Metrics, analysis, differential calculus, and optimization. Second edition. Advances in Design and Control, 22. Society for Industrial and Applied Mathematics (SIAM), Philadelphia, PA, 2011. 


%\bibitem{DoHe23}
%J.~D\"olz and F.~Henr\'iquez, Parametric Shape Holomorphy of Boundary Integral Operators with Applications, Preprint, (2023), arXiv:2305.19853

%\bibitem{FeAm22}
%F.~Feppon and H.~Ammari,  High order topological asymptotics: reconciling layer potentials and compound asymptotic expansions, {\it Multiscale Model. Simul.} {\bf 20} (2022), no. 3, 957--1001.

\bibitem{GiTr83}
D. Gilbarg {and} N.S. Trudinger,  {\it Elliptic partial
differential equations of second order. }Reprint of the 1998 edition. Classics in Mathematics. Springer-Verlag, Berlin, 2001. 


\bibitem{HaKr04}
H.~Haddar and R.~Kress, On the Fr\'echet derivative for obstacle scattering with an impedance boundary condition, {\it SIAM J. Appl. Math.} {\bf 65} (2004), no. 1,  194--208. 

%\bibitem{HeSc21}
%F.~Henr\'iquez and C.~Schwab, Shape holomorphy of the Calder\'on projector for the Laplacian in $\mathbb{R}^2$, {\it Integral Equations Operator Theory} {\bf 93} (2021), no. 4, Paper No. 43, 40 pp. 
 
\bibitem{HePi05}
A.~Henrot and M.~Pierre, {\it Variation et optimisation de formes}. Vol.~{\bf48} of
Math\'ematiques \& Applications (Berlin) [Mathematics \& Applications],
Springer, Berlin, 2005. %une analyse g\'eom\'etrique. [A geometric analysis].


%\bibitem{He95} 
%F.~Hettlich, Fr\'echet derivatives in inverse obstacle scattering, {\it Inverse Problems} {\bf11} (1995), no. 2,  371--382. 


\bibitem{HeRu01}
F. Hettlich {and} W. Rundell, 
Identification of a discontinuous source in the heat equation,
{\it Inverse Problems} {\bf 17} (2001), no.5, 1465--1482.

\bibitem{HiPh57}
E.~Hille and R.S.~Phillips, {\it Functional analysis and semi-groups}. American
  Mathematical Society Colloquium Publications, vol. 31, American Mathematical
  Society, Providence, R. I., 1957.
  
  
\bibitem{HoLeMi03}   
 S.~Hofmann, J.~Lewis, and M.~Mitrea, 
Spectral properties of parabolic layer potentials and transmission boundary problems in nonsmooth domains, {\it Illinois J. Math.} {\bf 47} (2003), no.4, 1345--1361.

 \bibitem{JeScZe17} 
C.~Jerez-Hanckes, C.~Schwab, and J.~Zech, Electromagnetic wave scattering by random surfaces: shape holomorphy, {\it Math. Models Methods Appl. Sci.} {\bf 27} (2017), no. 12, 2229--2259.

%\bibitem{Ke67}
%O.D.~Kellogg, {\it Foundations of potential theory}. Reprint from the first edition
%of 1929. Die Grundlehren der Mathematischen Wissenschaften, Band 31.
%Springer-Verlag, Berlin-New York, 1967.
  
\bibitem{Ki93} 
A.~Kirsch, The domain derivative and two applications in inverse scattering theory, {\it Inverse Problems} {\bf 9} (1993), no. 1, 81--96.
   
% \bibitem{Ki21}
%A.~Kirsch, {An introduction to the mathematical theory of inverse problems.} Applied Mathematical Sciences, 120. Springer, Cham, 2021.
  
  
%\bibitem{KrPa99}
%R.~Kress and L.~P\"aiv\"arinta, On the far field in obstacle scattering,
%{\it SIAM J. Appl. Math.} {\bf 59} (1999), no. 4,  1413--1426. 



%\bibitem{He06}
%A. Henrot, Extremum problems for eigenvalues of elliptic operators, Frontiers in Mathematics, Birkh\"auser Verlag, Basel, 2006. x+202 pp.
%
%  \bibitem{HePi05}
%A.~Henrot, M.~Pierre, Variation et optimisation de formes, Vol.~48 of
%  Math\'ematiques \& Applications (Berlin) [Mathematics \& Applications],
%  Springer, Berlin, 2005. xii+334 pp. %une analyse g\'eom\'etrique. [A geometric analysis].
%
%\bibitem{He82}
% D. Henry, Topics in nonlinear analysis, Universidade de Brasilia, Trabalho de Matematica {\bf 192} (1982).
% 
% 
% 

\bibitem{LaSoUr68}
O. A.~Lady\v{z}enskaja,  V.A.~Solonnikov, and N.N.~Ural'ceva, {\it  Linear and quasilinear equations of parabolic type.} (Russian) Translated from the Russian by S. Smith. Translations of Mathematical Monographs, {\bf 23} American Mathematical Society, Providence, R.I. 1968.

%
% 
%  \bibitem{La06}
%  M.~Lanza~de Cristoforis, A domain perturbation problem for the Poisson equation.  { \it Complex Var. Theory Appl.}
%  {\bf  50} (2005), no. 7-11, 851--867.

\bibitem{La07}
M.~Lanza~de Cristoforis, Perturbation problems in potential theory, a
functional analytic approach, {\it J. Appl. Funct. Anal.} { \bf 2} (2007), no. 3,  197--222.
  
 
%  \bibitem{LaLu17}
%M. Lanza de Cristoforis, P. Luzzini, Time dependent boundary norms for kernels and regularizing properties of the double-layer heat potential. {\it Eurasian Math. J.} {\bf 8} (2017), no. 1, 76--118.
% 
% 
%  \bibitem{LaLu19}
% M. Lanza de Cristoforis, P. Luzzini, Tangential derivatives and higher order regularizing properties of the double-layer heat potential, Analysis (Berlin), {\bf 8} (2019), no. 4, 167--193.
%
\bibitem{LaMu11}
M. Lanza de Cristoforis and P. Musolino, A perturbation result for periodic layer potentials of general second order differential operators with constant coefficients, 
{\it Far East J.  Math. Sci.}  {\bf 52} (2011), no. 1, 75--120.
%
%
%\bibitem{LaPr99}
%M. Lanza de Cristoforis, L. Preciso, On the analyticity of the Cauchy integral in Schauder spaces, {\it J. Integral Equations Appl.} {\bf 11} (1999),  no. 3, 363--391.


\bibitem{LaRo04}
M.~Lanza {de}~Cristoforis and L.~Rossi, Real analytic dependence of simple and
double-layer potentials upon perturbation of the support and of the density,
{\it J. Integral Equations Appl.} {\bf 16} (2004), no. 2, 137--174.

\bibitem{LaRo08}
M.~Lanza~de Cristoforis and L.~Rossi, Real analytic dependence of simple and double-layer potentials for the {H}elmholtz equation upon perturbation of the support and of the density, In {\it Analytic methods of analysis and differential equations: {AMADE} 2006}, Camb. Sci. Publ., Cambridge, 2008, pp. 193--220.
%
%\bibitem{Li96}
%G.M. Lieberman, Second order parabolic differential equations. World Scientific Publishing Co., Inc., River Edge, NJ, 1996. xii+439 pp.

%\bibitem{LuSiWa92} 	
%A. Lunardi , E. Sinestrari and W. von Wahl. A semigroup approach to the time 
%dependent parabolic initial-boundary value problem. {\it Differential and Integral Equations} 5 (1992), no. 6, 
%no. 6, 1275--1306.
%
%\bibitem{LuVe91}
%A. Lunardi, V. Vespri, HÃ¶lder regularity in variational parabolic non-homogeneous equations,  {\it Journal of Differential Equations}, {\bf 94} (1991) no. 1, 1--40.
%
%



%\bibitem{Lu19}
%P. Luzzini, Regularizing properties of the double-layer heat potential and shape analysis of a periodic problem. {\it Ph.D. Dissertation}, Universit\`a degli Studi di Padova (2019).
%

%\bibitem{Le12}
%F.~Le Lou\"er, On the Fr\'echet derivative in elastic obstacle scattering, {\it SIAM J. Appl. Math.} {\bf 72} (2012),  no. 5,  1493--1507.


\bibitem{Lu18}
P. Luzzini, Regularizing properties of space-periodic layer heat potentials and applications to boundary value problems in periodic domains, {\it Math.~Methods Appl. Sci.} {\bf 43} (2020), 5273--5294.


\bibitem{LuMu18}
P. Luzzini and P. Musolino, Periodic transmission problems for the heat equation, In {\it Integral methods in science and engineering}, 211--223, Birkh\"auser/Springer, Cham, 2019.
%
%\bibitem{LuMu20}
%P. Luzzini and P. Musolino, Perturbation analysis of the effective {transmission} of a periodic composite, {\it Netw. Heterog. Media} {\bf{15}} (2020), no. 4, 581--603.
%

%\bibitem{LuMuPu19}
%P. Luzzini, P. Musolino and R. Pukhtaievych, Shape analysis of the longitudinal flow along a periodic array of cylinders, {\it J. Math. Anal. Appl.} {\bf 477} (2019), no. 2, 1369--1395.

%\bibitem{NoSo13}
%A.~A. Novotny and J.~Soko{\l}{o}wski, {\it Topological derivatives in shape
%optimization}, Interaction of Mechanics and Mathematics, Springer, Heidelberg,  2013.

%\bibitem{PiHeJe23}
%J.~Pinto, F.~Henr\'iquez, C.~Jerez-Hanckes, Shape Holomorphy of Boundary Integral Operators on Multiple Open Arcs, Preprint, (2023), arXiv:2305.12202    
    
\bibitem{Pi02}
M.A.~Pinsky, {\it Introduction to Fourier analysis and wavelets.} Reprint of the 2002 original. Graduate Studies in Mathematics, {\bf 102}, American Mathematical Society, Providence, RI, 2009.

\bibitem{Po94}
R.~Potthast, 
Fr\'echet differentiability of boundary integral operators in inverse acoustic scattering, {\it Inverse Problems} {\bf 10} (1994), no. 2, 431--447. 


\bibitem{QiRiSaZh19}
T.~Qiu, A.~Rieder, F.J.~Sayas, and S.~Zhang, 
Time-domain boundary integral equation modeling of heat transmission problems, {\it Numer. Math.} {\bf 143} (2019), no.1, 223--259.

%\bibitem{Po96a} 
%R.~Potthast, Fr\'echet differentiability of the solution to the acoustic Neumann scattering problem with respect to the domain, {\it J. Inverse Ill-Posed Probl.}  {\bf 4} (1996),  no. 1,  67--84.
 
%\bibitem{Po96b}  
%R.~Potthast, Domain derivatives in electromagnetic scattering, {\it Math. Methods Appl. Sci.} {\bf 19} (1996), no. 15, 1157--1175. 

%\bibitem{Pu18}
%R.~Pukhtaievych, Effective conductivity of a periodic dilute composite with perfect contact and its series expansion, {\it Z. Angew. Math. Phys.} {\bf 69} (2018), no. 3, Paper No. 83, 22 pp. 

%\bibitem{Rud91}
%W.~Rudin, {\it Functional Analysis}, 2-nd edition, McGraw-Hill, Singapore, 1991.

\bibitem{SoZo92}
J.~Soko{\l}{o}wski and J.-P. Zol\'esio, {\it Introduction to shape optimization.
{S}hape sensitivity analysis}, Vol.~{\bf 16} of Springer Series in Computational Mathematics, Springer-Verlag, Berlin, 1992.



% \bibitem{NoSo13}
%A.A.~Novotny, J.~Soko\l owski,  Topological derivatives in shape optimization, Interaction of Mechanics and Mathematics, Springer, Heidelberg, 2013. xxii+412 pp.
% 
 
 
%\bibitem{NoSoZo19}  
%A.A.~Novotny, J.~Soko\l owski, A. \.{Z}ochowski,  Applications of the topological derivative method. With a foreword by Michel Delfour, Studies in Systems, Decision and Control, 188. Springer, Cham, 2019. xiv+212 pp. 








%
%\bibitem{Pi84}   
%O. Pironneau, Optimal shape design for elliptic systems. Springer Series in Computational Physics. Springer-Verlag, New York, 1984.  xii+168 pp.
%
%\bibitem{Po96I}
%R. Potthast, Domain derivatives in electromagnetic scattering, {\it Math. Methods Appl. Sci.} {\bf 19} (1996), no. 15, 1157--1175.
%
%\bibitem{Po94}
%R. Potthast, Fr\'echet differentiability of boundary integral operators in inverse acoustic scattering, {\it Inverse Problems} {\bf 10} (1994), no. 2, 431--447.
%
%\bibitem{Po96II}
%R. Potthast, Fr\'echet differentiability of the solution to the acoustic Neumann scattering problem with respect to the domain, {\it J. Inverse Ill-Posed Probl.}  {\bf 4} (1996), no. 1, 67--84.
%
%\bibitem{PoSt03}
% R. Potthast, I.G. Stratis,  On the domain derivative for scattering by impenetrable obstacles in chiral media. {\it IMA J. Appl. Math.} {\bf 68} (2003), no. 6, 621--635.
% 
 
% \bibitem{SoZo92}
% J. Sokolowski,  J.P. Zol\'esio, Shape sensitivity analysis. Springer Series in Computational Mathematics, 16. Springer-Verlag, Berlin, 1992. ii+250 pp.
%\bibitem{Va88}
%T. Valent, Boundary value problems of finite elasticity. Local theorems on existence, uniqueness, and analytic dependence on data. Springer Tracts in Natural Philosophy, 31. Springer-Verlag, New York, 1988. xii+191 pp.

\bibitem{Ve84}
G. Verchota,  Layer potentials and regularity for the Dirichlet problem for Laplace's equation in Lipschitz domains{,} \textit{J. Funct. Anal.} {\bf 59} (1984), no. 3, 572-611.

%\bibitem{Wu93}
% S. Wu, Analytic dependence of Riemann mappings for bounded domains and minimal surfaces, Comm. Pure Appl. Math. 46 (1993), 1303--1326.


\end{thebibliography}
\end{document}